\documentclass[english,a4paper]{amsart}
\usepackage[T1]{fontenc}
\usepackage[latin9]{inputenc}
\usepackage{amsmath}
\usepackage{amsthm}
\usepackage{amssymb}

\makeatletter
\theoremstyle{plain}
\newtheorem{thm}{\protect\theoremname}
  \theoremstyle{remark}
  \newtheorem{rem}[thm]{\protect\remarkname}
  \theoremstyle{plain}
  \newtheorem{lem}[thm]{\protect\lemmaname}
  \theoremstyle{plain}
  \newtheorem{prop}[thm]{\protect\propositionname}

\makeatother

\usepackage{babel}
  \providecommand{\lemmaname}{Lemma}
  \providecommand{\propositionname}{Proposition}
  \providecommand{\remarkname}{Remark}
\providecommand{\theoremname}{Theorem}

\begin{document}

\title[Blowing up solutions for supercritical Yamabe boundary problems]{Blowing up solutions for supercritical Yamabe problems on manifolds
with umbilic boundary}

\thanks{The first authors was supported by Gruppo Nazionale per l'Analisi Matematica, la Probabilit\`{a} e le loro Applicazioni (GNAMPA) of Istituto Nazionale di Alta Matematica (INdAM) and by project PRA from Univeristy of Pisa}

\author{Marco G. Ghimenti}
\address{M. Ghimenti, \newline Dipartimento di Matematica Universit\`a di Pisa
Largo B. Pontecorvo 5, 56126 Pisa, Italy}
\email{marco.ghimenti@unipi.it}

\author{Anna Maria Micheletti}
\address{A. M. Micheletti, \newline Dipartimento di Matematica Universit\`a di Pisa
Largo B. Pontecorvo 5, 56126 Pisa, Italy}
\email{a.micheletti@dma.unipi.it.}

\begin{abstract}
We build blowing-up solutions for a supercritical perturbation of
the Yamabe problem on manifolds with umbilic boundary provided the
dimension of the manifold is $n\ge8$ and that the Weyl tensor $W_{g}$
is not vanishing on $\partial M$.
\end{abstract}

\keywords{Umbilic boundary, Yamabe problem, Compactness, Stability}

\subjclass[2000]{35J65, 53C21}
\maketitle

\section{Introduction}

Let $(M,g)$ be a smooth compact Riemannian manifold of dimension
$n\ge3$ with a smooth boundary $\partial M$. A well known problem
in differential geometry is whether $(M,g)$ can be conformally deformed
in a constant scalar manifold with boundary of constant mean curvature.
When the boundary is empty this is called the Yamabe problem (see
\cite{Au,yam}) which has been completely solved by Aubin \cite{Au},
Schoen \cite{sch}, and Trundinger \cite{tru}. Escobar \cite{E92}
studied the problem in the context of manifolds with boundary and
gave an affirmative solution to the question in some cases. The remaining
cases were studied by Marques \cite{M1}, Almaraz \cite{A1}, Brendle
and Chen \cite{BC}, Mayer and Ndiaye \cite{MN}.

Once it is known that the problem admits solution, a natural question
about the compactness of the full set of solutions arises. Concerning
the Yamabe problem, a necessary condition is that the manifold is
not conformally equivalent to the standard sphere $\mathbb{S}^{n}$,
since the set of conformal transformation of the round sphere is not
compact. The problem of compactness has been studied by Schoen in
1988 \cite{s3} and by Brendle \cite{bre}, Brendle and Marques \cite{bre-mar},
Khuri Marques and Schoen \cite{khu-mar-sch} in the last years.

When the boundary of the manifold is not empty, a necessary condition
is that $M$ is not conformally equivalent to the standard ball $\mathbb{B}^{n}$.
Compactness for boundary Yamabe problem has been studied firstly by
Felli and Ould Ahmedou \cite{FO03}, Han and Li \cite{HL}, Almaraz
\cite{A3}. 

In this context the case of scalar flat metrics is particularly interesting
since it leads to study a linear equation in the interior with a critical
nonlinear Neumann-type boundary condition 
\begin{equation}
\left\{ \begin{array}{ll}
-\Delta_{g}u+\frac{n-2}{4(n-1)}R_{g}u=0 & \text{ on }M\\
\frac{\partial}{\partial\nu}u+\frac{n-2}{2}h_{g}u=(n-2)u^{\frac{n}{n-2}} & \text{ on }\partial M
\end{array}\right.\label{eq:Pnonpert}
\end{equation}
where $R_{g}$ is the scalar curvature of $M$, $h_{g}$ is the mean
curvature on $\partial M$ and $\nu$ is the outward normal to the
boundary. The geometric meaning of (\ref{eq:Pnonpert}) is that if
$u$ is a solution of (\ref{eq:Pnonpert}) the scalar curvature of
the conformal metric $\tilde{g}=u^{\frac{4}{n-2}}g$ is zero and the
mean curvature of $\tilde{g}$ on the boundary of $M$ is $n-2$.
The Yamabe boundary problem in the case of scalar flat metrics can
be also seen as the multidimensional version of the Riemann Mapping
Theorem. 

Concerning problem (\ref{eq:Pnonpert}), Felli and Ould Ahmedou in
\cite{FO03} have proved compactness when $M$ is locally conformally
flat and the boundary is umbilic and Almaraz in \cite{A3} has proved
compactness when $n\ge7$ and the trace free second fundamental form
is non zero everywhere on $\partial M$, that is any point of the
boundary is non umbilic. In \cite{KMW19} Kim Musso and Wei showed
that compactness continues to hold when $n=4$ and when $n=6,7$ and
the trace free second fundamental form is non zero everywhere on the
boundary. 

Very recently, compactness has been proved by the authors in \cite{GM20}
for manifold with umbilic boundary when $n>8$ and the Weyl tensor
of $M$ is everywhere non zero on the boundary $\partial M$. In \cite{GMsub}
the authors extend the compactness result to manifold of dimension
$n=6,7,8$, when the boundary is umbilic and the Weyl tensor of $M$
is everywhere non zero on $\partial M$. 

An example of non compactness is given for $n\ge25$ and manifold
with umbilic boundary in \cite{A2}. We recall that the boundary of
$M$ is called umbilic if the trace free second fundamental form of
$\partial M$ is zero everywhere on $\partial M$.

Another interesting question is the stability problem. One can ask
whether or not the compactness property is preserved under perturbation
of the equation. This is equivalent to having or not uniform a priori
estimates for solutions of the perturbed problem. 

In the following we consider the problem
\begin{equation}
\left\{ \begin{array}{ll}
L_{g}v=0 & \text{ on }M\\
B_{g}v+(n-2)v^{\frac{n}{n-2}+\varepsilon}=0 & \text{ on }\partial M
\end{array}\right.\label{eq:Pmain}
\end{equation}
where $\varepsilon$ is a positive real parameter, $L_{g}:=\Delta_{g}-\frac{n-2}{4(n-1)}R_{g}$
is the conformal Laplacian and $B_{g}=-\frac{\partial}{\partial\nu}v-\frac{n-2}{2}h_{g}(x)v$
is the conformal boundary operator. In the next we will use $a(x):=\frac{n-2}{4(n-1)}R_{g}$
to simplify the notation.

We study the question of stability of the problem (\ref{eq:Pnonpert}).
It is clear that the problem is not stable with respect to supercritical
perturbation of the nonlinearity if we are able to build solutions
$v_{\varepsilon}$ of the perturbed problem (\ref{eq:Pmain}) which
blow up at one point of the manifold as the parameter $\varepsilon$
goes to zero.

Our main result is the following
\begin{thm}
\label{almaraz} Let $M$ be a manifold of positive type with umbilic
boundary $\partial M$. Suppose $n\ge8$ and that the Weyl tensor
$W_{g}$ is not vanishing on $\partial M$. 

Then there exists a solution $v_{\varepsilon}$ of (\ref{eq:Pmain})
such that $v_{\varepsilon}$ blows up when $\varepsilon\rightarrow0^{+}$.
\end{thm}
Here $M$ of positive type means that there exists $C>0$ such that
\[
Q(u):=\frac{\int\limits _{M}\left(|\nabla u|^{2}+\frac{n-2}{4(n-1)}R_{g}u^{2}\right)dv_{g}+\int\limits _{\partial M}\frac{n-2}{2}h_{g}u^{2}d\sigma_{g}}{\left(\int\limits _{\partial M}|u|^{\frac{2(n-1)}{n-2}}d\sigma_{g}\right)^{\frac{n-2}{n-1}}}\ge C\text{ for any }u\in H^{1}(M)\smallsetminus\left\{ 0\right\} .
\]
We remark that this assumption on the positivity of $Q$ in natural
when we address to compactness questions in Yamabe problems since
if ${\displaystyle \inf_{u\in H^{1}(M)\smallsetminus\{0\}}Q(u)\le0}$,
then the solution of Yamabe problem is unique.

The stability of problem (\ref{eq:Pnonpert}) with respect to the
principal quantity of the boundary term has been studied in a series
of paper by the authors and by Pistoia, both in the case of non umbilic
boundary and in the case of umbilic boundary with Weyl tensor non
vanishing on the boundary. Firstly, they studied what happens linearly
perturbing the mean curvature term. This problem present a strong
analogy to the Yamabe problem when perturbing the scalar curvature
term (see, on this topic \cite{dru,dru-heb} and the references therein).
In fact, we have that the set of solutions is compact -and hence (\ref{eq:Pnonpert})
is stable- perturbing the mean curvature from below while we construct
a blowing up sequence when the perturbation is everywhere positive
on the boundary, and for a class of perturbation which are positive
in at least one point on the boundary. The result of compactness is
dealt in \cite{GMdcds} both in umbilic and non umbilic case while
for the construction of blowing up sequences for umbilic boundary
manifold we refer to \cite{GMP19}.

Concerning the exponent of the nonlinearity, all the compactness results
hold for $p\le\frac{n}{n-2}$, so the Yamabe boundary problem is stable
from below with respect to the critical exponent, while, in the present
paper we have that small perturbations above the critical exponent
imply blowing up solution when $n\ge8$, the boundary $\partial M$
is umbilic and the Weyl tensor is non vanishing on $\partial M$. 

As a final remark, we notice that in \cite{GMP19} we ask that the
manifold is umbilic, the Weyl tensor is never vanishing and that $n\ge11.$
The assumption on the dimension in this paper is technical, since
we ask some integrability condition when performing the Ljapounov
Schmidt procedure. In deed, in the present paper we perform more effective
computations in Lemma \ref{lem:R}. This method could be applied verbatim
in paper \cite{GMP19}, so we can reformulate the main result in dimension
$n\ge8$. 

\section{Preliminaries}

We recall here a series of preliminary result which are useful for
our result.

Since the manifold is of positive type, then 
\[
\left\langle \left\langle u,v\right\rangle \right\rangle _{g}=\int_{M}(\nabla_{g}u\nabla_{g}v+auv)d\mu_{g}+\frac{n-2}{2}\int_{\partial M}h_{g}uvd\sigma_{g}
\]
is an equivalent scalar product in $H_{g}^{1}$, which induces to
the equivalent norm $\|\cdot\|_{g}$. 

We define the exponent 
\[
s_{\varepsilon}=\frac{2(n-1)}{n-2}+n\varepsilon
\]
and the Banach space $\mathcal{H}_{g}:=H^{1}(M)\cap L^{s_{\varepsilon}}(\partial M)$
endowed with norm $\|u\|_{\mathcal{H}_{g}}=\|u\|_{g}+|u|_{L_{g}^{s_{\varepsilon}}(\partial M)}.$By
trace theorems, we have the following inclusion $W^{1,\tau}(M)\subset L^{t}(\partial M)$
for $t\le\tau\frac{n-1}{n-\tau}$. 

We recall the following result, by Nittka \cite[Th. 3.14]{Nit}
\begin{rem}
\label{rem:Nit}Let $\frac{2n}{n+2}\le q<\frac{n}{2}$, $r>0$. Then
there exists a constant $c$ such that if $f_{0}\in L^{q+r}(\Omega)$,
$\beta$ bounded and measurable and $g\in L^{\frac{(n-1)q}{n-q}+r}(\partial\Omega)$
and $u\in H^{1}(\Omega)$ is the unique weak solution of 
\[
\left\{ \begin{array}{ll}
Lu=f_{0} & \text{ on }\Omega\\
\frac{\partial}{\partial\nu}u+\beta u=g & \text{ on }\partial\Omega
\end{array}\right.
\]
where $L$ is a strictly elliptic second order operator, then 
\[
u\in L^{\frac{nq}{n-2q}}(\Omega),\ \left.u\right|_{\partial\Omega}L^{\frac{(n-1)q}{n-2q}}(\partial\Omega)\text{ and}
\]
\[
|u|_{L^{\frac{nq}{n-2q}}(\Omega)}+|u|_{L^{\frac{(n-1)q}{n-2q}}(\partial\Omega)}\le\left|f_{0}\right|_{L^{q+r}(\Omega)}+\left|g\right|_{L^{\frac{(n-1)q}{n-q}+r}(\partial\Omega)}
\]
\end{rem}
We consider $i:H^{1}(M)\rightarrow L^{\frac{2(n-1)}{n-2}}(\partial M)$
and its adjoint with respect to $\left\langle \left\langle \cdot,\cdot\right\rangle \right\rangle _{g}$
\[
i_{g}^{*}:L^{\frac{2(n-1)}{n}}(\partial M)\rightarrow H^{1}(M)
\]
defined by
\[
\left\langle \left\langle \varphi,i_{g}^{*}(f)\right\rangle \right\rangle _{g}=\int_{\partial M}\varphi fd\sigma_{g}\text{ for all }\varphi\in H^{1}
\]
so that $v=i_{g}^{*}(g)$ is the weak solution of the problem
\begin{equation}
\left\{ \begin{array}{ll}
-\Delta_{g}v+a(x)v=0 & \text{ on }M\\
\frac{\partial}{\partial\nu}v+\frac{n-2}{2}h_{g}(x)v=f & \text{ on }\partial M
\end{array}\right..\label{eq:istellasopra}
\end{equation}

By \cite[Th. 3.14]{Nit} (see Remark \ref{rem:Nit}) we have that,
if $v\in H^{1}$ is a solution of (\ref{eq:istellasopra}), then for
$\frac{2n}{n+2}\le q<\frac{n}{2}$ and $r>0$ it holds
\begin{equation}
|v|_{L^{\frac{(n-1)q}{n-2q}}(\partial M)}=|i_{g}^{*}(f)|_{L^{\frac{(n-1)q}{n-2q}}(\partial M)}\le|f|_{L^{\frac{(n-1)q}{n-q}+r}(\partial M)}.\label{eq:nittka}
\end{equation}
By this result, we can choose $q,r$ such that
\begin{equation}
\frac{(n-1)q}{n-2q}=\frac{2(n-1)}{n-2}+n\varepsilon\text{ and }\frac{(n-1)q}{n-q}+r=\frac{2(n-1)+n(n-2)\varepsilon}{n+(n-2)\varepsilon}\label{eq:nittka1}
\end{equation}
that is 
\[
q=\frac{2n+n^{2}\left(\frac{n-2}{n-1}\right)\varepsilon}{n+2+2n\left(\frac{n-2}{n-1}\right)\varepsilon}\text{ and }r=\frac{2(n-1)+n(n-2)\varepsilon}{n+(n-2)\varepsilon}-\frac{2(n-1)+n(n-2)\varepsilon}{n+\left(n-2\right)\left(\frac{n}{n-1}\right)\varepsilon}.
\]
 Set $f_{\varepsilon}(v)=(n-2)\left(v^{+}\right)^{\frac{n}{n-2}+\varepsilon}$,
we have that, if $v\in L_{g}^{\frac{2(n-1)}{n-2}+n\varepsilon}(\partial M)$,
then $f_{\varepsilon}(v)\in L_{g}^{\frac{2(n-1)+n(n-2)\varepsilon}{n+\varepsilon(n-2)}}(\partial M)$
and, in light of (\ref{eq:nittka}), also $i_{g}^{*}\left(f_{\varepsilon}(v)\right)\in L_{g}^{\frac{2(n-1)}{n-2}+n\varepsilon}(\partial M)$.

Thus we can recast then Problem (\ref{eq:Pmain}) as 
\begin{equation}
v=i_{g}^{*}\left(f_{\varepsilon}(v)\right),\ v\in\mathcal{H}_{g}.\label{eq:P*}
\end{equation}

The problem has also a variational structure: we can associate to
Problem (\ref{eq:Pmain}) the following functional, which is well
defined on $\mathcal{H}_{g}$. 
\begin{align}
J_{\varepsilon,g}(v):= & \frac{1}{2}\int_{M}|\nabla_{g}v|^{2}+av^{2}d\mu_{g}+\frac{n-2}{4}\int_{\partial M}h_{g}v^{2}d\sigma_{g}\label{eq:Jeps}\\
 & -\frac{(n-2)^{2}}{2(n-1)+\varepsilon(n-2)}\int_{\partial M}\left(v^{+}\right)^{\frac{2(n-1)}{n-2}+\varepsilon}d\sigma_{g}.\nonumber 
\end{align}

\begin{rem}
Since $\partial M$ is umbilic for any $q\in\partial M$, there exists
a metric $\tilde{g}_{q}=\tilde{g}$, conformal to $g$, $\tilde{g}_{q}=\Lambda_{q}^{\frac{4}{n-2}}g_{q}$
such that
\begin{equation}
|\text{det}\tilde{g}_{q}(y)|=1+O(|y|^{N})\label{eq:|g|}
\end{equation}
\begin{equation}
|\tilde{h}_{ij}(y)|=o(|y^{3}|)\label{eq:hij}
\end{equation}
\begin{align}
\tilde{g}^{ij}(y)= & \delta_{ij}+\frac{1}{3}\bar{R}_{ikjl}y_{k}y_{l}+R_{ninj}y_{n}^{2}\label{eq:gij}\\
 & +\frac{1}{6}\bar{R}_{ikjl,m}y_{k}y_{l}y_{m}+R_{ninj,k}y_{n}^{2}y_{k}+\frac{1}{3}R_{ninj,n}y_{n}^{3}\nonumber \\
 & +\left(\frac{1}{20}\bar{R}_{ikjl,mp}+\frac{1}{15}\bar{R}_{iksl}\bar{R}_{jmsp}\right)y_{k}y_{l}y_{m}y_{p}\nonumber \\
 & +\left(\frac{1}{2}R_{ninj,kl}+\frac{1}{3}\text{Sym}_{ij}(\bar{R}_{iksl}R_{nsnj})\right)y_{n}^{2}y_{k}y_{l}\nonumber \\
 & +\frac{1}{3}R_{ninj,nk}y_{n}^{3}y_{k}+\frac{1}{12}\left(R_{ninj,nn}+8R_{nins}R_{nsnj}\right)y_{n}^{4}+O(|y|^{5})\nonumber 
\end{align}
\begin{equation}
\bar{R}_{\tilde{g}_{q}}(y)=O(|y|^{2})\text{ and }\partial_{ii}^{2}\bar{R}_{\tilde{g}_{q}}(q)=-\frac{1}{6}|\bar{W}(q)|^{2}\label{eq:Rii}
\end{equation}
\begin{equation}
\bar{R}_{kl}(q)=R_{nn}(q)=R_{nk}(q)=0\label{eq:Ricci}
\end{equation}
uniformly with respect to $q\in\partial M$ and $y\in T_{q}(M)$.
Also,we have $\Lambda_{q}(q)=1$ and $\nabla\Lambda_{q}(q)=0$. This
results are contained in \cite{M1,KMW}. Here $\tilde{h}_{ij}$ is
the tensor of the second fundamental form referred to the metric $\tilde{g}$.
\end{rem}
The conformal Laplacian and the conformal boundary operator transform
under the change of metric $\tilde{g}_{q}=\Lambda_{q}^{\frac{4}{n-2}}g_{q}$
as follows:
\begin{align*}
L_{\tilde{g}_{q}}\varphi & =\Lambda_{q}^{-\frac{n+2}{n-2}}L_{g}(\Lambda_{q}\varphi)\\
B_{\tilde{g}_{q}}\varphi & =\Lambda_{q}^{-\frac{n}{n-2}}B_{g}(\Lambda_{q}\varphi).
\end{align*}
By these transformations we have that $v:=\Lambda_{q}u$ is a positive
solution of (\ref{eq:Pmain}), if and only if $u$ is a positive solution
of 
\begin{equation}
\left\{ \begin{array}{cc}
L_{\tilde{g}_{q}}u=0 & \text{ in }M\\
B_{\tilde{g}_{q}}u+(n-2)\Lambda_{q}^{\varepsilon}u^{\frac{n}{n-2}+\varepsilon}=0 & \text{ on }\partial M
\end{array}\right.\label{eq:Ptilde}
\end{equation}
From now on we set $\tilde{f}_{\varepsilon}(u)=(n-2)\Lambda_{q}^{\varepsilon}\left(u^{+}\right)^{\frac{n}{n-2}+\varepsilon}$.

Furthermore we have 
\[
\left\langle \left\langle \Lambda_{q}u,\Lambda_{q}v\right\rangle \right\rangle _{g}=\left\langle \left\langle u,v\right\rangle \right\rangle _{\tilde{g}}
\]
and, consequently, 
\[
\|\Lambda_{q}u\|_{g}=\|u\|_{\tilde{g}}.
\]
In addition, we have that $\Lambda_{q}u\in L_{g}^{s_{\varepsilon}}$
if and only if $u\in L_{\tilde{g}}^{s_{\varepsilon}}$, so $\Lambda_{q}u\in\mathcal{H}_{g}$
if and only if $u\in\mathcal{H}_{\tilde{g}}$. Finally, we can define
the functional $J_{\varepsilon,\tilde{g}}$ associated to (\ref{eq:Ptilde}),
as 
\begin{align*}
J_{\varepsilon,\tilde{g}}(u):= & \frac{1}{2}\int_{M}|\nabla_{\tilde{g}}u|^{2}+\tilde{a}u^{2}d\mu_{\tilde{g}}+\frac{n-2}{4}\int_{\partial M}h_{\tilde{g}}v^{2}d\sigma_{\tilde{g}}\\
 & -\frac{(n-2)^{2}}{2(n-1)+\varepsilon(n-2)}\Lambda_{q}\int_{\partial M}\left(u^{+}\right)^{\frac{2(n-1)}{n-2}+\varepsilon}d\sigma_{\tilde{g}},
\end{align*}
where $\tilde{a}=\frac{n-2}{4(n-1)}R_{\tilde{g}}$ , and we get
\[
J_{\varepsilon,g}(\Lambda_{q}u)=J_{\varepsilon,\tilde{g}}(u).
\]
so we can always switch between metrics $g$ and $\tilde{g}$ , and
this will be useful in the next. In the Section \ref{sec:reduction}
we emphasize other equivalences of the same kind. As a last remark,
we notice also that a solution of (\ref{eq:Ptilde}) can be expressed
by means of $i_{\tilde{g}}^{*}$, in fact $u$ solves (\ref{eq:Ptilde})
if and only if 
\[
u=i_{\tilde{g}}^{*}(\tilde{f}_{\varepsilon}(u)).
\]

\section{The finite dimensional reduction.\label{sec:reduction}}

Given $q\in\partial M$ and $\psi_{q}^{\partial}:\mathbb{R}_{+}^{n}\rightarrow M$
the Fermi coordinates in a neighborhood of $q$; we define 
\begin{align*}
W_{\delta,q}(\xi) & =U_{\delta}\left(\left(\psi_{q}^{\partial}\right)^{-1}(\xi)\right)\chi\left(\left(\psi_{q}^{\partial}\right)^{-1}(\xi)\right)=\\
 & =\frac{1}{\delta^{\frac{n-2}{2}}}U\left(\frac{y}{\delta}\right)\chi(y)=\frac{1}{\delta^{\frac{n-2}{2}}}U\left(x\right)\chi(\delta x)
\end{align*}
where $y=(z,t)$, with $z\in\mathbb{R}^{n-1}$ and $t\ge0$, $\delta x=y=\left(\psi_{q}^{\partial}\right)^{-1}(\xi)$
and $\chi$ is a radial cut off function, with support in ball of
radius $R$.

Here $U_{\delta}(y)=\frac{1}{\delta^{\frac{n-2}{2}}}U\left(\frac{y}{\delta}\right)$
is the one parameter family of solution of the problem 
\begin{equation}
\left\{ \begin{array}{ccc}
-\Delta U_{\delta}=0 &  & \text{on }\mathbb{R}_{+}^{n};\\
\frac{\partial U_{\delta}}{\partial t}=-(n-2)U_{\delta}^{\frac{n}{n-2}} &  & \text{on \ensuremath{\partial}}\mathbb{R}_{+}^{n}.
\end{array}\right.\label{eq:Udelta}
\end{equation}
and ${\displaystyle U(z,t):=\frac{1}{\left[(1+t)^{2}+|z|^{2}\right]^{\frac{n-2}{2}}}}$
is the standard bubble in $\mathbb{R}_{+}^{n}$.\\

Now, let us consider the linearized problem 
\begin{equation}
\left\{ \begin{array}{ccc}
 & -\Delta\phi=0 & \text{on }\mathbb{R}_{+}^{n},\\
 & \frac{\partial\phi}{\partial t}+nU^{\frac{2}{n-2}}\phi=0 & \text{on \ensuremath{\partial}}\mathbb{R}_{+}^{n},\\
 & \phi\in H^{1}(\mathbb{R}_{+}^{n}).
\end{array}\right.\label{eq:linearizzato}
\end{equation}
and it is well know that  every solution of (\ref{eq:linearizzato})
is a linear combination of the functions $j_{1},\dots,j_{n}$ defined
by . 
\begin{eqnarray}
j_{i}=\frac{\partial U}{\partial x_{i}},\ i=1,\dots n-1 &  & j_{n}=\frac{n-2}{2}U+\sum_{i=1}^{n}y_{i}\frac{\partial U}{\partial y_{i}}.\label{eq:sol-linearizzato}
\end{eqnarray}

Given $q\in\partial M$ we define, for $b=1,\dots,n$ 
\[
Z_{\delta,q}^{b}(\xi)=\frac{1}{\delta^{\frac{n-2}{2}}}j_{b}\left(\frac{1}{\delta}\left(\psi_{q}^{\partial}\right)^{-1}(\xi)\right)\chi\left(\left(\psi_{q}^{\partial}\right)^{-1}(\xi)\right)
\]
and we decompose $H^{1}(M)$ in the direct sum of the following two
subspaces 
\begin{align*}
\tilde{K}_{\delta,q} & =\text{Span}\left\langle \Lambda_{q}Z_{\delta,q}^{1},\dots,\Lambda_{q}Z_{\delta,q}^{n}\right\rangle \\
\tilde{K}_{\delta,q}^{\bot} & =\left\{ \varphi\in H^{1}(M)\ :\ \left\langle \left\langle \varphi,\Lambda_{q}Z_{\delta,q}^{b}\right\rangle \right\rangle _{g}=0,\ b=1,\dots,n\right\} 
\end{align*}
and we define the projections 
\[
\tilde{\Pi}=H^{1}(M)\rightarrow\tilde{K}_{\delta,q}\text{ and }\tilde{\Pi}^{\bot}=H^{1}(M)\rightarrow\tilde{K}_{\delta,q}^{\bot}.
\]

In order to give a good ansatz on the shape of the solution we need
to introduce the function $v_{q}:\mathbb{R}_{+}^{n}\rightarrow\mathbb{R}$
which is a solution of the linear problem 
\begin{equation}
\left\{ \begin{array}{ccc}
-\Delta v_{q}=\left[\frac{1}{3}\bar{R}_{ijkl}(q)y_{k}y_{l}+R_{ninj}(q)y_{n}^{2}\right]\partial_{ij}^{2}U &  & \text{on }\mathbb{R}_{+}^{n}\\
\frac{\partial v}{\partial y_{n}}=-nU^{\frac{2}{n-2}}v_{q} &  & \text{on }\partial\mathbb{R}_{+}^{n}
\end{array}\right.\label{eq:vqdef}
\end{equation}
This function is a key tool for several estimates in what follows.
In fact, a good choice of $v_{q}$ we allow us to get the correct
size of the remainder term in the finite dimensional reduction (Lemma
\ref{lem:R}). 
\begin{rem}
\label{rem:vq}There exists a unique $v_{q}:\mathbb{R}_{+}^{n}\rightarrow\mathbb{R}$
solution of the problem (\ref{eq:vqdef}) $L^{2}(\mathbb{R}_{+}^{n})$-ortogonal
to $j_{b}$ for all $b=1,\dots,n$. Moreover it holds
\begin{equation}
|\nabla^{\tau}v_{q}(y)|\le C(1+|y|)^{4-\tau-n}\text{ for }\tau=0,1,2,\label{eq:gradvq}
\end{equation}
\begin{equation}
\int_{\partial\mathbb{R}_{+}^{n}}U^{\frac{n}{n-2}}(t,z)v_{q}(t,z)dz=0\label{eq:Uvq}
\end{equation}
and 
\begin{equation}
\int_{\partial\mathbb{R}_{+}^{n}}v_{q}(t,z)\Delta v_{q}(t,z)dz\le0,\label{new}
\end{equation}
where $y\in\mathbb{R}_{+}^{n}$, $y=(t,z)$ with $t\ge0$ and $z\in\mathbb{R}^{n-1}$.
In addition, the map $q\mapsto v_{q}$ is in $C^{2}(\partial M)$.
\end{rem}
The proof of this remark can be found in \cite[Lemma 3]{GMP19} and
will be omitted.

At this point, given $q\in\partial M$ we define, similarly to $W_{\delta,q}$,
the function 
\[
V_{\delta,q}(\xi)=\frac{1}{\delta^{\frac{n-2}{2}}}v_{q}\left(\frac{1}{\delta}\left(\psi_{q}^{\partial}\right)^{-1}(\xi)\right)\chi\left(\left(\psi_{q}^{\partial}\right)^{-1}(\xi)\right)
\]
and 
\[
\left(v_{q}\right)_{\delta}(y)=\frac{1}{\delta^{\frac{n-2}{2}}}v_{q}\left(\frac{y}{\delta}\right),
\]
where $v_{q}$ is chosen as in Remark \ref{rem:vq}.

We look for solution of (\ref{eq:P*}) having the form 
\[
v=\Lambda_{q}u=\tilde{W}_{\delta,q}+\delta^{2}\tilde{V}_{\delta,q}+\tilde{\phi}\text{ with }\tilde{\phi}\in\tilde{K}_{\delta,q}^{\bot}\cap\mathcal{H}.
\]
where we used the intuitive notation
\[
\tilde{W}_{\delta,q}=\Lambda_{q}W_{\delta,q}\ \ \tilde{V}_{\delta,q}=\Lambda_{q}V_{\delta,q}\text{ and }\tilde{\phi}=\Lambda_{q}\phi
\]
We can rewrite, in light of the previous orthogonal decomposition,
Problem (\ref{eq:P*}) (and so Problem (\ref{eq:Pmain})) as 
\begin{align}
\tilde{\Pi}\left\{ \tilde{W}_{\delta,q}+\delta^{2}\tilde{V}_{\delta,q}+\tilde{\phi}-i_{g}^{*}\left[f_{\varepsilon}(\tilde{W}_{\delta,q}+\delta^{2}\tilde{V}_{\delta,q}+\tilde{\phi})\right]\right\}  & =0\label{eq:P-K}\\
\tilde{\Pi}^{\bot}\left\{ \tilde{W}_{\delta,q}+\delta^{2}\tilde{V}_{\delta,q}+\tilde{\phi}-i_{g}^{*}\left[f_{\varepsilon}(\tilde{W}_{\delta,q}+\delta^{2}\tilde{V}_{\delta,q}+\tilde{\phi})\right]\right\}  & =0.\label{eq:P-Kort}
\end{align}
We stress out than we can proceed in analogous way in the manifold
$(M,\tilde{g})$. In this case we should define 
\begin{align*}
K_{\delta,q} & =\text{Span}\left\langle Z_{\delta,q}^{1},\dots,Z_{\delta,q}^{n}\right\rangle \\
K_{\delta,q}^{\bot} & =\left\{ \varphi\in H^{1}(M)\ :\ \left\langle \left\langle \varphi,Z_{\delta,q}^{b}\right\rangle \right\rangle _{\tilde{g}}=0,\ b=1,\dots,n\right\} ,
\end{align*}
and we should ask that $u=W_{\delta,q}+\delta^{2}V_{\delta,q}+\phi$,
recasting (\ref{eq:Ptilde}) as the couple of equations
\begin{align}
\Pi\left\{ W_{\delta,q}+\delta^{2}V_{\delta,q}+\phi-i_{\tilde{g}}^{*}\left[\tilde{f}_{\varepsilon}(W_{\delta,q}+\delta^{2}V_{\delta,q}+\phi)\right]\right\}  & =0\label{eq:Pi-1}\\
\Pi^{\bot}\left\{ W_{\delta,q}+\delta^{2}V_{\delta,q}+\phi-i_{\tilde{g}}^{*}\left[\tilde{f}_{\varepsilon}(W_{\delta,q}+\delta^{2}V_{\delta,q}+\phi)\right]\right\}  & =0.\label{eq:Pibot-1}
\end{align}
Roughly speaking, we are allowed to move the tilde symbol from solutions
to problems and vice versa, so in any moment we can choose in which
metric and with which functional it is more convenient to work. 

Coming back to problem (\ref{eq:P-Kort}), we define the linear operator
$L:\tilde{K}_{\delta,q}^{\bot}\cap\mathcal{H}_{g}\rightarrow\tilde{K}_{\delta,q}^{\bot}\cap\mathcal{H}_{g}$
as
\begin{equation}
L(\tilde{\phi})=\tilde{\Pi}^{\bot}\left\{ \tilde{\phi}-i_{g}^{*}\left(f_{\varepsilon}'(\tilde{W}_{\delta,q}+\delta^{2}\tilde{V}_{\delta,q})[\tilde{\phi}]\right)\right\} ,\label{eq:defL-1}
\end{equation}
a nonlinear term $N(\tilde{\Phi})$ and a remainder term $R$ as 
\begin{align}
N(\tilde{\phi})= & \tilde{\Pi}^{\bot}\left\{ i_{g}^{*}\left(f_{\varepsilon}(\tilde{W}_{\delta,q}+\delta^{2}\tilde{V}_{\delta,q}+\tilde{\phi})-f_{\varepsilon}(\tilde{W}_{\delta,q}+\delta^{2}\tilde{V}_{\delta,q})-f'_{\varepsilon}(\tilde{W}_{\delta,q}+\delta^{2}\tilde{V}_{\delta,q})[\tilde{\phi}]\right)\right\} \label{eq:defN-1}\\
R= & \tilde{\Pi}^{\bot}\left\{ i_{g}^{*}\left(f_{\varepsilon}(\tilde{W}_{\delta,q}+\delta^{2}\tilde{V}_{\delta,q})\right)-\tilde{W}_{\delta,q}-\delta^{2}\tilde{V}_{\delta,q}\right\} ,\label{eq:defR-1}
\end{align}
so equation (\ref{eq:P-Kort}) becomes

\[
L(\tilde{\phi})=N(\tilde{\phi})+R.
\]
The rest of this section is devoted to show that for any choice of
$\delta,q$ a solution $\tilde{\phi}$ of (\ref{eq:P-Kort}) exists.
\begin{lem}
\label{lem:L}Let $\delta=\varepsilon^{\frac{1}{4}}\lambda$. For
$a,b\in\mathbb{R}$, $0<a<b$ there exists a positive constant $C_{0}=C_{0}(a,b)$
such that, for $\varepsilon$ small, for any $q\in\partial M$, for
any $\lambda\in[a,b]$ and for any $\phi\in K_{\delta,q}^{\bot}\cap\mathcal{H}$
there holds
\[
\|L_{\delta,q}(\phi)\|_{\mathcal{H}}\ge C_{0}\|\phi\|_{\mathcal{H}}.
\]
\end{lem}
\begin{proof}
The proof of this Lemma is very similar to the proof of \cite[Lemma 2]{GMP16}
and will be omitted.
\end{proof}
\begin{lem}
\label{lem:R}It holds
\[
\|R\|_{\mathcal{H}_{g}}=\left\{ \begin{array}{cc}
\delta^{-O^{+}(\varepsilon)}\left\{ O\left(\delta^{3}\log\delta\right)+O(\varepsilon\log\delta)+O(\varepsilon)\right\}  & \text{ if }n=8\\
O\left(\delta^{3}\right)+\delta^{-O^{+}(\varepsilon)}\left\{ O(\varepsilon\log\delta)+O(\varepsilon)\right\}  & \text{\text{ if }}n>8
\end{array}\right.
\]
where $0<O^{+}(\varepsilon)<C\varepsilon$ for some positive constant
$C$. In addition, with the choice $\delta=\varepsilon^{\frac{1}{4}}\lambda$
we have that 
\[
\|R\|_{\mathcal{H}_{g}}=\left\{ \begin{array}{cc}
O\left(\varepsilon^{\frac{3}{4}}\log\varepsilon\right) & \text{ if }n=8\\
O\left(\varepsilon^{\frac{3}{4}}\right) & \text{\text{ if }}n>8
\end{array}\right..
\]
\end{lem}
\begin{proof}
\textbf{Step 1.} It holds 
\begin{equation}
\|R\|_{g}=\left\{ \begin{array}{cc}
O\left(\delta^{3}\log\delta\right)+O(\varepsilon\log\delta)+O(\varepsilon) & \text{ if }n=8\\
O\left(\delta^{3}\right)+O(\varepsilon\log\delta)+O(\varepsilon) & \text{\text{ if }}n>8
\end{array}\right.\label{eq:Rg}
\end{equation}
We have 
\begin{align*}
\left\Vert R\right\Vert _{g} & \le\left\Vert i_{g}^{*}\left(f_{\varepsilon}(\tilde{W}_{\delta,q}+\delta^{2}\tilde{V}_{\delta,q})\right)-i_{g}^{*}\left(f_{0}(\tilde{W}_{\delta,q}+\delta^{2}\tilde{V}_{\delta,q})\right)\right\Vert _{g}\\
 & +\left\Vert i_{g}^{*}\left(f_{0}(\tilde{W}_{\delta,q}+\delta^{2}\tilde{V}_{\delta,q})\right)-\tilde{W}_{\delta,q}-\delta^{2}\tilde{V}_{\delta,q}\right\Vert _{g},
\end{align*}
and we start by estimating the second term. By definition of $i_{g}^{*}$
there exists $\Gamma=i_{g}^{*}\left(f_{0}(\tilde{W}_{\delta,q}+\delta^{2}\tilde{V}_{\delta,q})\right)$,
that is a function $\Gamma$ solving
\begin{equation}
\left\{ \begin{array}{ll}
-\Delta_{g}\Gamma+a(x)\Gamma=0 & \text{ on }M\\
\frac{\partial}{\partial\nu}\Gamma+\frac{n-2}{2}h_{g}(x)\Gamma=f_{0}(\tilde{W}_{\delta,q}+\delta^{2}\tilde{V}_{\delta,q}) & \text{ on }\partial M
\end{array}\right..\label{eq:gamma}
\end{equation}
So we have
\begin{multline*}
\left\Vert i_{g}^{*}\left(f_{0}(\tilde{W}_{\delta,q}+\delta^{2}\tilde{V}_{\delta,q}\right)-\tilde{W}_{\delta,q}-\delta^{2}\tilde{V}_{\delta,q}\right\Vert _{g}^{2}=\|\Gamma-\tilde{W}_{\delta,q}-\delta^{2}\tilde{V}_{\delta,q}\|_{g}^{2}\\
=\int_{M}\left[-\Delta_{g}(\Gamma-\tilde{W}_{\delta,q}-\delta^{2}\tilde{V}_{\delta,q})+a(\Gamma-\tilde{W}_{\delta,q}-\delta^{2}\tilde{V}_{\delta,q})\right](\Gamma-\tilde{W}_{\delta,q}-\delta^{2}\tilde{V}_{\delta,q})d\mu_{g}\\
+\int_{\partial M}h_{g}(\Gamma-\tilde{W}_{\delta,q}-\delta^{2}\tilde{V}_{\delta,q})^{2}d\sigma_{g}\\
+\int_{\partial M}\left[\frac{\partial}{\partial\nu}(\Gamma-\tilde{W}_{\delta,q}-\delta^{2}\tilde{V}_{\delta,q})\right](\Gamma-\tilde{W}_{\delta,q}-\delta^{2}\tilde{V}_{\delta,q})d\sigma_{g}\\
=\int_{M}\left[\Delta_{g}(\tilde{W}_{\delta,q}+\delta^{2}\tilde{V}_{\delta,q})-a(\tilde{W}_{\delta,q}+\delta^{2}\tilde{V}_{\delta,q})\right](\Gamma-\tilde{W}_{\delta,q}-\delta^{2}\tilde{V}_{\delta,q})d\mu_{g}\\
-\int_{\partial M}h_{g}(\tilde{W}_{\delta,q}+\delta^{2}\tilde{V}_{\delta,q})(\Gamma-\tilde{W}_{\delta,q}-\delta^{2}\tilde{V}_{\delta,q})d\sigma_{g}\\
+\int_{\partial M}\left[f_{0}(\tilde{W}_{\delta,q}+\delta^{2}\tilde{V}_{\delta,q})-\frac{\partial}{\partial\nu}(\tilde{W}_{\delta,q}+\delta^{2}\tilde{V}_{\delta,q})\right](\Gamma-\tilde{W}_{\delta,q}-\delta^{2}\tilde{V}_{\delta,q})d\sigma_{g}\\
=:I_{1}+I_{2}+I_{3}.
\end{multline*}
We have
\begin{align*}
I_{2} & =\int_{\partial M}h_{\tilde{g}}(W_{\delta,q}+\delta^{2}V_{\delta,q})(\Lambda_{q}^{-1}R)d\sigma_{\tilde{g}}\\
 & \le C|h_{\tilde{g}}(W_{\delta,q}+\delta^{2}V_{\delta,q})|_{L_{\tilde{g}}^{\frac{2(n-1)}{n}}(\partial M)}\|\Lambda_{q}^{-1}R\|_{\tilde{g}}
\end{align*}
Set $B_{1/\delta}^{n-1}=\left\{ z\in\mathbb{R}^{n-1},\ |z|\le1/\delta\right\} $,
we have
\begin{align*}
|h_{\tilde{g}}(W_{\delta,q}+\delta^{2}V_{\delta,q})|_{L_{\tilde{g}}^{\frac{2(n-1)}{n}}(\partial M)} & =O(\delta)\left|h_{\tilde{g}}(\delta z)(U(z)-\delta^{2}v_{q}(z))\right|_{L^{\frac{2(n-1)}{n}}(B_{1/\delta}^{n-1})}.
\end{align*}
 Since $z\le1/\delta$ we have that $\delta(1+|z|)=O(1)$. We recall
that $|\nabla^{\tau}v_{q}(y)|\le C(1+|y|)^{4-\tau-n}$ by (\ref{eq:gradvq})
and that $|\nabla^{\tau}U(y)|\le C(1+|y|)^{2-\tau-n}$ for $\tau=0,1,2$.
By \ref{eq:hij} we have also that $h_{\tilde{g}_{q}}(q)=h_{\tilde{g}_{q},i}(q)=h_{\tilde{g}_{q},ik}(q)=0$,
so, 
\begin{align}
|h_{\tilde{g}}(W_{\delta,q}+\delta^{2}V_{\delta,q})|_{L_{\tilde{g}}^{\frac{2(n-1)}{n}}(\partial M)} & =O(\delta)\left|h_{\tilde{g}}(\delta z)(1+|z|)^{2-n}\right|_{L^{\frac{2(n-1)}{n}}(B_{1/\delta}^{n-1})}\nonumber \\
 & =O(\delta^{3})\left||z|^{2}(1+|z|)^{2-n}\right|_{L^{\frac{2(n-1)}{n}}(B_{1/\delta}^{n-1})}\nonumber \\
 & =O(\delta^{3}),\label{eq:hg-resto}
\end{align}
since $|z|^{2}(1+|z|)^{2-n}\le(1+|z|)^{4-n}$ and $\left|(1+|z|)^{4-n}\right|_{L^{\frac{2(n-1)}{n}}(B_{1/\delta}^{n-1})}$
is bounded when $n>8$ or $\left|(1+|z|)^{4-n}\right|_{L^{\frac{2(n-1)}{n}}(B_{1/\delta}^{n-1})}=O(\log\delta)$
when $n=8$. Thus
\[
I_{2}=O(\delta^{3})\|\Lambda_{q}^{-1}R\|_{\tilde{g}}=\left\{ \begin{array}{cc}
O(\delta^{3}\log\delta)\|R\|_{g} & \text{ if }n=8\\
O(\delta^{3})\|R\|_{g} & \text{ if }n>8
\end{array}\right..
\]
For $I_{1}$ we proceed in a similar way, having

\begin{align*}
I_{1} & =\int_{M}\left[\Delta_{\tilde{g}}(W_{\delta,q}+\delta^{2}V_{\delta,q})-\tilde{a}(W_{\delta,q}+\delta^{2}V_{\delta,q})\right](\Lambda_{q}^{-1}R)d\mu_{\tilde{g}}\\
 & \le\left|\Delta_{\tilde{g}}(W_{\delta,q}+\delta^{2}V_{\delta,q})-\tilde{a}(W_{\delta,q}+\delta^{2}V_{\delta,q})\right|_{L_{\tilde{g}}^{\frac{2n}{n+2}}(M)}\|\Lambda_{q}^{-1}R\|_{\tilde{g}}.
\end{align*}
 Set $B_{1/\delta}^{n}=\left\{ z\in\mathbb{R}^{n},\ |z|\le1/\delta\right\} $.
By \cite[page 1609]{M1}, we have $R_{\tilde{g}}(0)=0$, so we get
\begin{align*}
\left|\tilde{a}(W_{\delta,q}+\delta^{2}V_{\delta,q})\right|_{L_{\tilde{g}}^{\frac{2n}{n+2}}(M)} & =O(\delta^{2})\left|R_{\tilde{g}}(\delta x)(U(x)+\delta^{2}v_{q}(x))\right|_{L_{\tilde{g}}^{\frac{2n}{n+2}}(B_{1/\delta}^{n})}\\
 & O(\delta^{3})\left||x|(1+|x|)^{2-n}\right|_{L_{\tilde{g}}^{\frac{2n}{n+2}}(B_{1/\delta}^{n})}=\left\{ \begin{array}{cc}
O(\delta^{3}\log\delta) & \text{ if }n=8\\
O(\delta^{3}) & \text{ if }n>8
\end{array}\right.
\end{align*}
since $\left|(1+|x|)^{3-n}\right|_{L_{\tilde{g}}^{\frac{2n}{n+2}}(B_{1/\delta}^{n})}$
is bounded when $n>8$ and $\left|(1+|x|)^{3-n}\right|_{L_{\tilde{g}}^{\frac{2n}{n+2}}(B_{1/\delta}^{n})}=O(\log\delta)$
when $n=8$.

For the Laplacian term, in local charts we have
\begin{align*}
\Delta_{\tilde{g}_{q}} & =\Delta_{\text{euc}}+[\tilde{g}_{q}^{ij}(y)-\delta_{ij}]\partial_{ij}^{2}\\
 & +\left[\partial_{i}\tilde{g}_{q}^{ij}(y)+\frac{\tilde{g}_{q}^{ij}(y)\partial_{i}|\tilde{g}_{q}|^{\frac{1}{2}}(y)}{|\tilde{g}_{q}|^{\frac{1}{2}}(y)}\right]\partial_{j}+\frac{\partial_{n}|\tilde{g}_{q}|^{\frac{1}{2}}(y)}{|\tilde{g}_{q}|^{\frac{1}{2}}(y)}\partial_{n}.
\end{align*}
Thus by the expansion of the metric given in (\ref{eq:|g|}), (\ref{eq:gij})
and since $v_{q}$ solves (\ref{eq:vqdef}) and $\Delta_{\text{euc}}U=0$,
we have that
\begin{multline}
\left|\Delta_{\tilde{g}}(W_{\delta,q}+\delta^{2}V_{\delta,q})\right|_{L_{\tilde{g}}^{\frac{2n}{n+2}}(M)}=O(1)\left|\Delta_{\tilde{g}}(U+\delta^{2}v_{q})\right|_{L^{\frac{2n}{n+2}}(B_{1/\delta}^{n})}\\
=O(1)\Biggl|\Delta U+[\tilde{g}_{q}^{ij}(\delta x)-\delta_{ij}]\partial_{ij}^{2}U+\delta^{2}\Delta_{\text{euc}}v_{q}+\delta^{2}[\tilde{g}_{q}^{ij}(\delta x)-\delta_{ij}]\partial_{ij}^{2}v_{q}\\
+\left[\partial_{i}\tilde{g}_{q}^{ij}(\delta x)+\frac{\tilde{g}_{q}^{ij}(\delta x)\partial_{i}|\tilde{g}_{q}|^{\frac{1}{2}}(\delta x)}{|\tilde{g}_{q}|^{\frac{1}{2}}(\delta x)}\right]\partial_{j}(U+\delta^{2}v_{q})+\frac{\partial_{n}|\tilde{g}_{q}|^{\frac{1}{2}}(\delta x)}{|\tilde{g}_{q}|^{\frac{1}{2}}(\delta x)}\partial_{n}(U+\delta^{2}v_{q})\Biggr|_{L^{\frac{2n}{n+2}}(B_{1/\delta}^{n})}\\
=O(1)\left|\delta^{3}|x|^{3}\partial_{ij}^{2}U+\delta^{4}|x|^{2}\partial_{ij}^{2}v_{q}+\delta^{3}|x|^{2}\partial_{j}(U+\delta^{2}v_{q})+\delta^{3}|x|^{2}\partial_{n}(U+\delta^{2}v_{q})\right|_{L^{\frac{2n}{n+2}}(B_{1/\delta}^{n})}\\
=O(1)\left|\delta^{3}(1+|x|)^{3-n}+\delta^{4}(1+|x|)^{4-n}\right|_{L^{\frac{2n}{n+2}}(B_{1/\delta}^{n})}\\
=O(\delta^{3})\left|(1+|x|)^{3-n}\right|_{L^{\frac{2n}{n+2}}(B_{1/\delta}^{n})}=\left\{ \begin{array}{cc}
O(\delta^{3}\log\delta) & \text{ if }n=8\\
O(\delta^{3}) & \text{ if }n>8
\end{array}\right.,\label{eq:Delta-resto}
\end{multline}
and we conclude that 
\[
I_{1}=\left\{ \begin{array}{cc}
O(\delta^{3}\log\delta)\|R\|_{g} & \text{ if }n=8\\
O(\delta^{3})\|R\|_{g} & \text{ if }n>8
\end{array}\right..
\]
For the last integral $I_{3}$ we have
\begin{align}
I_{3}\le & C\left|(n-2)\left((W_{\delta,q}+\delta^{2}V_{\delta,q})^{+}\right)^{\frac{n}{n-2}}-\frac{\partial}{\partial\nu}(W_{\delta,q}+\delta^{2}V_{\delta,q})\right|_{L_{\tilde{g}}^{\frac{2(n-1)}{n}}(\partial M)}\left\Vert R\right\Vert _{g}\nonumber \\
\le & C(n-2)\left|\left((W_{\delta,q}+\delta^{2}V_{\delta,q})^{+}\right)^{\frac{n}{n-2}}-\left(W_{\delta,q}\right)^{\frac{n}{n-2}}-\delta^{2}\frac{\partial}{\partial\nu}V_{\delta,q}\right|_{L_{\tilde{g}}^{\frac{2(n-1)}{n}}(\partial M)}\left\Vert R\right\Vert _{g}\nonumber \\
 & +C\left|(n-2)\left(W_{\delta,q}\right)^{\frac{n}{n-2}}-\frac{\partial}{\partial\nu}W_{\delta,q}\right|_{L_{\tilde{g}}^{\frac{2(n-1)}{n}}(\partial M)}\left\Vert R\right\Vert _{g}.\label{eq:bordo1}
\end{align}
Since $U$ is a solution of (\ref{eq:Udelta}) one can easily obtain
\begin{equation}
\left|(n-2)\left(W_{\delta,q}\right)^{\frac{n}{n-2}}-\frac{\partial}{\partial\nu}W_{\delta,q}\right|_{L_{\tilde{g}}^{\frac{2(n-1)}{n}}(\partial M)}=O(\delta^{3}).\label{eq:bordo2}
\end{equation}
Finally we have, using  (\ref{eq:vqdef}), and expanding $\left((U+\delta^{2}v_{\delta,q})^{+}\right)^{\frac{n}{n-2}}$
near $U$,
\begin{multline}
\left|\left((W_{\delta,q}+\delta^{2}V_{\delta,q})^{+}\right)^{\frac{n}{n-2}}-\left(W_{\delta,q}\right)^{\frac{n}{n-2}}-\delta^{2}\frac{\partial}{\partial\nu}V_{\delta,q}\right|_{L_{\tilde{g}}^{\frac{2(n-1)}{n}}(\partial M)}\\
=O(1)\left|\left((U+\delta^{2}v_{\delta,q})^{+}\right)^{\frac{n}{n-2}}-U^{\frac{n}{n-2}}-\delta^{2}\frac{\partial}{\partial\nu}v_{q}\right|_{L^{\frac{2(n-1)}{n}}(B_{1/\delta}^{n-1})}\\
=O(\delta^{2})\left|\left((U+\theta\delta^{2}v_{\delta,q})^{+}\right)^{\frac{2}{n-2}}v_{q}-U^{\frac{n}{n-2}}v_{q}\right|_{L^{\frac{2(n-1)}{n}}(B_{1/\delta}^{n-1})}.\label{eq:bordo3}
\end{multline}
By the decay estimates (\ref{eq:gradvq}) we have that $U+\theta\delta v_{q}>0$
in $B_{1/\delta}^{n-1}$ provided $\delta$ small enough. So, expanding
again we have 
\begin{multline}
O(\delta^{2})\left|\left((U+\theta\delta^{2}v_{\delta,q})^{+}\right)^{\frac{2}{n-2}}v_{q}-U^{\frac{n}{n-2}}v_{q}\right|_{L^{\frac{2(n-1)}{n}}(B_{1/\delta}^{n-1})}\\
=O(\delta^{3})\left|\delta\left((U+\theta_{1}\delta^{2}v_{\delta,q})^{+}\right)^{\frac{4-n}{n-2}}v_{q}^{2}\right|_{L^{\frac{2(n-1)}{n}}(B_{1/\delta}^{n-1})}\\
=O(\delta^{3})\left|\delta\left(1+|y|\right)^{4-n}\right|_{L^{\frac{2(n-1)}{n}}(B_{1/\delta}^{n-1})}\\
=O(\delta^{3})\left|\left(1+|y|\right)^{3-n}\right|_{L^{\frac{2(n-1)}{n}}(B_{1/\delta}^{n-1})}=O(\delta^{3})\label{eq:bordo4}
\end{multline}
since $n\ge8$ and we get $I_{3}=O(\delta^{3})\left\Vert R\right\Vert _{g}$
and, consequently,

\[
\left\Vert i_{g}^{*}\left(f_{0}(\tilde{W}_{\delta,q}+\delta^{2}\tilde{V}_{\delta,q}\right)-\tilde{W}_{\delta,q}-\delta^{2}\tilde{V}_{\delta,q}\right\Vert _{g}^{2}=O(\delta^{3})\|R\|_{g}.
\]
To conclude the first part of the proof we estimate the term 
\[
\left\Vert i_{g}^{*}\left(f_{\varepsilon}(\tilde{W}_{\delta,q}+\delta^{2}\tilde{V}_{\delta,q})\right)-i_{g}^{*}\left(f_{0}(\tilde{W}_{\delta,q}+\delta^{2}\tilde{V}_{\delta,q})\right)\right\Vert _{g}
\]
It is useful to recall the following Taylor expansions with respect
to $\varepsilon$
\begin{align}
U^{\varepsilon} & =1+\varepsilon\ln U+\frac{1}{2}\varepsilon^{2}\ln^{2}U+o(\varepsilon^{2})\label{eq:Uallaeps}\\
\delta^{-\varepsilon\frac{n-2}{2}} & =1-\varepsilon\frac{n-2}{2}\ln\delta+\varepsilon^{2}\frac{(n-2)^{2}}{8}\ln^{2}\delta+o(\varepsilon^{2}\ln^{2}\delta)\label{eq:deltaallaeps}
\end{align}
We have that, recalling that $\Lambda_{q}(0)=0$,
\begin{multline}
\left\Vert i_{g}^{*}\left(f_{\varepsilon}(\tilde{W}_{\delta,q}+\delta^{2}\tilde{V}_{\delta,q}\right)-i_{g}^{*}\left(f_{0}(\tilde{W}_{\delta,q}+\delta^{2}\tilde{V}_{\delta,q}\right)\right\Vert _{g}\\
=\left\Vert i_{\tilde{g}}^{*}\left(\tilde{f}_{\varepsilon}(W_{\delta,q}+\delta^{2}V_{\delta,q}\right)-i_{\tilde{g}}^{*}\left(f_{0}(W_{\delta,q}+\delta^{2}V_{\delta,q}\right)\right\Vert _{\tilde{g}}\\
\le C\left|\Lambda_{q}^{\varepsilon}\left(W_{\delta,q}+\delta^{2}V_{\delta,q}\right)^{\frac{n}{n-2}+\varepsilon}-\left(W_{\delta,q}+\delta^{2}V_{\delta,q}\right)^{\frac{n}{n-2}}\right|_{L_{\tilde{g}}^{\frac{2(n-1)}{n}}(\partial M)}\\
=O(1)\left|\Lambda_{q}^{\varepsilon}(\delta y)\left(U+\delta^{2}v_{q}\right)^{\frac{n}{n-2}+\varepsilon}-\left(U+\delta^{2}v_{q}\right)^{\frac{n}{n-2}}\right|_{L^{\frac{2(n-1)}{n}}(B_{1/\delta}^{n-1})}\\
=O(1)\left|\left(\frac{\Lambda_{q}^{\varepsilon}(\delta y)}{\delta^{\varepsilon\frac{n-2}{2}}}(U+\delta^{2}v_{q})^{\varepsilon}-1\right)(U+\delta^{2}v_{q})^{\frac{n}{n-2}}\right|_{L^{\frac{2(n-1)}{n}}(B_{1/\delta}^{n-1})}\\
\le O(1)\left|\left(-\frac{n-2}{2}\varepsilon\ln\delta+\varepsilon\ln(U+\delta^{2}v_{q})+O(\varepsilon^{2}\ln\delta)\right)U^{\frac{n}{n-2}}\right|_{L^{\frac{2(n-1)}{n}}(B_{1/\delta}^{n-1})}\\
=O(\varepsilon\log\delta)+O(\varepsilon),\label{eq:feps-f0}
\end{multline}
and we have proved (\ref{eq:Rg}).

\textbf{Step 2.} It holds
\begin{equation}
|R_{\varepsilon,\delta,q}|_{L^{s_{\varepsilon}}(\partial M)}=\left\{ \begin{array}{cc}
\delta^{-O^{+}(\varepsilon)}\left\{ O\left(\delta^{3}\log\delta\right)+O(\varepsilon\log\delta)+O(\varepsilon)\right\}  & \text{ if }n=8\\
O\left(\delta^{3}\right)+\delta^{-O^{+}(\varepsilon)}\left\{ O(\varepsilon\log\delta)+O(\varepsilon)\right\}  & \text{\text{ if }}n>8
\end{array}\right..\label{eq:Rseps}
\end{equation}
As in the previous case we consider 
\begin{align*}
|R|_{L_{g}^{s_{\varepsilon}}(\partial M)}\le & \left|i_{g}^{*}\left(f_{\varepsilon}(\tilde{W}_{\delta,q}+\delta^{2}\tilde{V}_{\delta,q})\right)-i_{g}^{*}\left(f_{0}(\tilde{W}_{\delta,q}+\delta^{2}\tilde{V}_{\delta,q})\right)\right|_{L_{g}^{s_{\varepsilon}}(\partial M)}\\
 & +\left|i_{g}^{*}\left(f_{0}(\tilde{W}_{\delta,q}+\delta^{2}\tilde{V}_{\delta,q})\right)-\tilde{W}_{\delta,q}(x)-\delta^{2}\tilde{V}_{\delta,q}\right|_{L_{g}^{s_{\varepsilon}}(\partial M)}
\end{align*}
and we start estimating the second term. Taking again $\Gamma=i_{g}^{*}\left(f_{0}(W_{\delta,q}+\delta V_{\delta,q}\right)$
the solution of (\ref{eq:gamma}), the function $\Gamma-W_{\delta,q}-\delta V_{\delta,q}$
solves the problem
\[
\left\{ \begin{array}{ll}
-\Delta_{g}(\Gamma-\tilde{W}_{\delta,q}-\delta^{2}\tilde{V}_{\delta,q})+a(x)(\Gamma-\tilde{W}_{\delta,q}-\delta^{2}\tilde{V}_{\delta,q})\\
=-\Delta_{g}(\tilde{W}_{\delta,q}+\delta^{2}\tilde{V}_{\delta,q})+a(x)(\tilde{W}_{\delta,q}+\delta^{2}\tilde{V}_{\delta,q}) & \text{ on }M\\
\\
\frac{\partial}{\partial\nu}(\Gamma-\tilde{W}_{\delta,q}-\delta^{2}\tilde{V}_{\delta,q})=f_{0}(\tilde{W}_{\delta,q}+\delta^{2}\tilde{V}_{\delta,q})-\frac{\partial}{\partial\nu}(\tilde{W}_{\delta,q}+\delta^{2}\tilde{V}_{\delta,q}) & \text{ on }\partial M
\end{array}\right..
\]
We choose $q=\frac{2n+n^{2}\left(\frac{n-2}{n-1}\right)\varepsilon}{n+2+2n\left(\frac{n-2}{n-1}\right)\varepsilon}$
and $r=\varepsilon$, so, by Remark \ref{rem:Nit}, we get 
\begin{align*}
|\Gamma-\tilde{W}_{\delta,q}-\delta^{2}\tilde{V}_{\delta,q}|_{L_{g}^{s_{\varepsilon}}(\partial M)}\le & |-\Delta_{g}(\tilde{W}_{\delta,q}+\delta^{2}\tilde{V}_{\delta,q})+a(x)(\tilde{W}_{\delta,q}+\delta^{2}\tilde{V}_{\delta,q})|_{L_{g}^{q+\varepsilon}(M)}\\
 & +\left|f_{0}(\tilde{W}_{\delta,q}+\delta^{2}\tilde{V}_{\delta,q})-\frac{\partial}{\partial\nu}(\tilde{W}_{\delta,q}+\delta^{2}\tilde{V}_{\delta,q})\right|_{L_{g}^{\frac{(n-1)q}{n-q}+\varepsilon}(\partial M)}.
\end{align*}
We remark that with our choice we can write $q=\frac{2n}{n+2}+O^{+}(\varepsilon)$,
$\frac{1}{q+\varepsilon}=\frac{n+2}{2n}-O^{+}(\varepsilon)$ and $\frac{(n-1)q}{n-q}+\varepsilon=\frac{2(n-1)}{n}+O^{+}(\varepsilon)$
where $0<O^{+}(\varepsilon)<C\varepsilon$ for some positive constant
$C$. 

We proceed as in the first part, obtaining
\[
|a(x)(\tilde{W}_{\delta,q}+\delta^{2}\tilde{V}_{\delta,q})|_{L_{g}^{q+\varepsilon}(M)}=O\left(\delta^{3-O^{+}(\varepsilon)}\right)\left|(1+|x|)^{3-n}\right|_{L_{\tilde{g}}^{\frac{2n}{n+2}+O^{+}(\varepsilon)}(B_{1/\delta}^{n})}.
\]
At this point we have that, if $n>8$, $\left|(1+|x|)^{3-n}\right|_{L_{\tilde{g}}^{\frac{2n}{n+2}+O^{+}(\varepsilon)}(B_{1/\delta}^{n})}=O(1)$,
while, if $n=8$, $\left|(1+|x|)^{3-n}\right|_{L_{\tilde{g}}^{\frac{2n}{n+2}+O^{+}(\varepsilon)}(B_{1/\delta}^{n})}=O\left(\delta^{-O^{+}(\varepsilon)}\log\delta\right)$.
So we get 
\[
|a(x)(\tilde{W}_{\delta,q}+\delta^{2}\tilde{V}_{\delta,q})|_{L_{g}^{q+\varepsilon}(M)}=\left\{ \begin{array}{cc}
O\left(\delta^{3-O^{+}(\varepsilon)}\log\delta\right) & \text{ if }n=8\\
O\left(\delta^{3}\right) & \text{\text{ if }}n>8
\end{array}\right..
\]
In the same spirit one can check that
\begin{align*}
|-\Delta_{g}(\tilde{W}_{\delta,q}+\delta^{2}\tilde{V}_{\delta,q})|_{L_{g}^{q+\varepsilon}(M)} & =\left\{ \begin{array}{cc}
O\left(\delta^{3-O^{+}(\varepsilon)}\log\delta\right) & \text{ if }n=8\\
O\left(\delta^{3}\right) & \text{\text{ if }}n>8
\end{array}\right.;\\
\left|f_{0}(\tilde{W}_{\delta,q}+\delta^{2}\tilde{V}_{\delta,q})-\frac{\partial}{\partial\nu}(\tilde{W}_{\delta,q}+\delta^{2}\tilde{V}_{\delta,q})\right|_{L_{g}^{\frac{(n-1)q}{n-q}+\varepsilon}(\partial M)} & =\left\{ \begin{array}{cc}
O\left(\delta^{3-O^{+}(\varepsilon)}\log\delta\right) & \text{ if }n=8\\
O\left(\delta^{3}\right) & \text{\text{ if }}n>8
\end{array}\right..
\end{align*}
To finish the proof we estimate 
\[
\left|i_{g}^{*}\left(f_{\varepsilon}(\tilde{W}_{\delta,q}+\delta^{2}\tilde{V}_{\delta,q})\right)-i_{g}^{*}\left(f_{0}(\tilde{W}_{\delta,q}+\delta^{2}\tilde{V}_{\delta,q})\right)\right|_{L_{g}^{s_{\varepsilon}}(\partial M)}
\]
Again, by Remark \ref{rem:Nit}, we have
\begin{multline*}
\left|i_{g}^{*}\left(f_{\varepsilon}(\tilde{W}_{\delta,q}+\delta^{2}\tilde{V}_{\delta,q})\right)-i_{g}^{*}\left(f_{0}(\tilde{W}_{\delta,q}+\delta^{2}\tilde{V}_{\delta,q})\right)\right|_{L_{g}^{s_{\varepsilon}}(\partial M)}\\
\le\left|f_{\varepsilon}(\tilde{W}_{\delta,q}+\delta^{2}\tilde{V}_{\delta,q})-f_{0}(\tilde{W}_{\delta,q}+\delta^{2}\tilde{V}_{\delta,q})\right|_{L_{g}^{\frac{2(n-1)}{n}+O^{+}(\varepsilon)}(\partial M)}
\end{multline*}
and, proceeding as in (\ref{eq:feps-f0}) we obtain
\begin{multline*}
\left|f_{\varepsilon}(\tilde{W}_{\delta,q}+\delta^{2}\tilde{V}_{\delta,q})-f_{0}(\tilde{W}_{\delta,q}+\delta^{2}\tilde{V}_{\delta,q})\right|_{L_{g}^{\frac{2(n-1)}{n}+O^{+}(\varepsilon)}(\partial M)}\\
=\delta^{-O^{+}(\varepsilon)}\left\{ O(\varepsilon\left|\ln\delta\right|)+O(\varepsilon)\right\} ,
\end{multline*}
and we have proved (\ref{eq:Rseps}). 

The last claim follows by direct computation.
\end{proof}
\begin{prop}
\label{prop:phi}Let $\delta=\lambda\varepsilon^{\frac{1}{4}}$ For
$a,b\in\mathbb{R}$, $0<a<b$ there exists a positive constant $C=C(a,b)$
such that, for $\varepsilon$ small, for any $q\in\partial M$ and
for any $\lambda\in[a,b]$ there exists a unique $\phi_{\delta,q}$
which solves (\ref{eq:P-K}) with
\[
\|\phi_{\delta,q}\|_{\mathcal{H}_{g}}=\left\{ \begin{array}{cc}
O\left(\varepsilon^{\frac{3}{4}}\log\varepsilon\right) & \text{ if }n=8\\
O\left(\varepsilon^{\frac{3}{4}}\right) & \text{\text{ if }}n>8
\end{array}\right..
\]
Moreover the map $q\mapsto\phi_{\delta,q}$ is a $C^{1}(\partial M,\mathcal{H}_{g})$
map.
\end{prop}
\begin{proof}
First we prove that the nonlinear operator $N$ defined (\ref{eq:defN-1})
is a contraction on a suitable ball of $\mathcal{H}$. Recalling that
\[
\|N(\phi_{1})-N(\phi_{2})\|_{\mathcal{H}}=\|N(\phi_{1})-N(\phi_{2})\|_{H}+|N(\phi_{1})-N(\phi_{2})|_{L^{s_{\varepsilon}}(\partial M)}
\]
we estimate the two right hand side terms separately.

By the continuity of $i^{*}:L^{\frac{2(n-1)}{n}}(\partial M)\rightarrow H$,
and by Lagrange theorem we have
\begin{multline*}
\|N(\phi_{1})-N(\phi_{2})\|_{H}\\
\le\left\Vert \left(f'_{\varepsilon}\left(W_{\delta,q}+\theta\phi_{1}+(1-\theta)\phi_{2}+\delta V_{\delta,q}\right)-f'_{\varepsilon}(W_{\delta,q}+\delta V_{\delta,q})\right)[\phi_{1}-\phi_{2}]\right\Vert _{L^{\frac{2(n-1)}{n}}(\partial M)}
\end{multline*}
and, since $|\phi_{1}-\phi_{2}|^{\frac{2(n-1)}{n}}\in L^{\frac{n}{n-2}}(\partial M)$
and $|f_{\varepsilon}'(\cdot)|^{\frac{2(n-1)}{n}}\in L^{\frac{n}{2}}(\partial M)$,
we have 
\begin{multline*}
\|N(\phi_{1})-N(\phi_{2})\|_{H}\\
\le\left\Vert \left(f'_{\varepsilon}\left(W_{\delta,q}+\theta\phi_{1}+(1-\theta)\phi_{2}+\delta V_{\delta,q}\right)-f'_{\varepsilon}(W_{\delta,q})+\delta V_{\delta,q}\right)\right\Vert _{L^{\frac{2(n-1)}{2}}(\partial M)}\|\phi_{1}-\phi_{2}\|_{H}\\
=\gamma\|\phi_{1}-\phi_{2}\|_{H}
\end{multline*}
where we can choose
\[
\gamma:=\left\Vert \left(f_{\varepsilon}'\left(W_{\delta,q}+\theta\phi_{1}+(1-\theta)\phi_{2}+\delta V_{\delta,q}\right)-f_{\varepsilon}'(W_{\delta,q}+\delta V_{\delta,q})\right)\right\Vert _{L^{\frac{2(n-1)}{n-2}}(\partial M)}<1,
\]
provided $\|\phi_{1}\|_{H}$ and $\|\phi_{2}\|_{H}$ sufficiently
small. 

For the second term we argue in a similar way and, recalling that,
by (\ref{eq:nittka}), $|i^{*}(g)|_{L^{s_{\varepsilon}}(\partial M)}\le|g|_{L^{\frac{2(n-1)+n(n-2)\varepsilon}{n+(n-2)\varepsilon}}(\partial M)}$,
we have
\begin{multline*}
|N(\phi_{1})-N(\phi_{2})|_{L^{s_{\varepsilon}}(\partial M)}\\
\le\left|\left(f'_{\varepsilon}\left(W_{\delta,q}+\theta\phi_{1}+(1-\theta)\phi_{2}+\delta V_{\delta,q}\right)-f'_{\varepsilon}(W_{\delta,q}+\delta V_{\delta,q})\right)[\phi_{1}-\phi_{2}]\right|_{L^{\frac{2(n-1)+n(n-2)\varepsilon}{n+(n-2)\varepsilon}}(\partial M)}
\end{multline*}
Since $\phi_{1},\phi_{2},W_{\delta,q}V_{\delta,q}\in L^{s_{\varepsilon}}$
we have that $|\phi_{1}-\phi_{2}|^{\frac{2(n-1)+n(n-2)\varepsilon}{n+(n-2)\varepsilon}}\in L^{\frac{n+(n-2)\varepsilon}{n-2}}(\partial M)$
and $|f'(\cdot)|^{\frac{2(n-1)+n(n-2)\varepsilon}{n+(n-2)\varepsilon}}\in L^{\frac{n+(n-2)\varepsilon}{2+(n-2)\varepsilon}}(\partial M)$.
So we conclude as above that we can choose $\left|\phi_{1}\right|_{L^{s_{\varepsilon}}(\partial M)}$,
$\left|\phi_{2}\right|_{L^{s_{\varepsilon}}(\partial M)}$ sufficiently
small in order to get
\[
|N(\phi_{1})-N(\phi_{2})|_{L^{s_{\varepsilon}}(\partial M)}\le\gamma\left|\phi_{1}-\phi_{2}\right|_{L^{s_{\varepsilon}}(\partial M)}.
\]
So 
\[
\|N(\phi_{1})-N(\phi_{2})\|_{\mathcal{H}}\le\gamma\|\phi_{1}-\phi_{2}\|_{\mathcal{H}}
\]
with $\gamma<1$, provided $\|\phi_{1}\|_{\mathcal{H}}$, $\|\phi_{2}\|_{\mathcal{H}}$
small enough.

With the same strategy it is possible to prove that if $\|\phi\|_{\mathcal{H}}$
is sufficiently small there exists $\bar{\gamma}<1$ such that $\|N(\phi)\|_{\mathcal{H}}\le\bar{\gamma}\|\phi\|_{\mathcal{H}}$. 

At this point, recalling Lemma \ref{lem:L} and Lemma \ref{lem:R},
it is not difficult to prove that there exists a constant $C>0$ such
that, if $n>8$ and $\|\phi\|_{\mathcal{H}}\le C\varepsilon^{\frac{3}{4}}$
then the map
\[
T(\phi):=L^{-1}(N(\phi)+R)
\]
is a contraction from the ball $\|\phi\|_{\mathcal{H}}\le C\varepsilon^{\frac{3}{4}}$
in itself. We proceed analogously for $n=8$ and we get the first
claim by the Contraction Mapping Theorem. The regularity claim can
be proven via the Implicit Function Theorem.
\end{proof}

\section{The reduced problem}

For any choice of $(\delta,q)$ Proposition \ref{prop:phi} states
the we can solve the infinite dimensional problem (\ref{eq:P-Kort}).
Now, set $\delta=\lambda\varepsilon^{\frac{1}{4}}$, we look for a
critical point for the functional $J_{\varepsilon,g}$ having the
form $\tilde{W}_{\lambda\varepsilon^{\frac{1}{4}},q}+\lambda^{2}\varepsilon^{\frac{1}{2}}\tilde{V}_{\lambda\varepsilon^{\frac{1}{4}},q}+\tilde{\phi}_{\lambda\varepsilon^{\frac{1}{4}},q}$.We
define the function
\begin{align*}
I_{\varepsilon}(\lambda,q): & =J_{\varepsilon,g}\left(\tilde{W}_{\lambda\varepsilon^{\frac{1}{4}},q}+\lambda^{2}\varepsilon^{\frac{1}{2}}\tilde{V}_{\lambda\varepsilon^{\frac{1}{4}},q}+\tilde{\phi}_{\lambda\varepsilon^{\frac{1}{4}},q}\right)\\
I_{\varepsilon}: & [a,b]\times\partial M\rightarrow\mathbb{R}
\end{align*}
 Which will useful in the next
\begin{lem}
\label{lem:JWpiuPhi}Assume $n\ge8$ and $\delta=\lambda\varepsilon^{\frac{1}{4}}$.
It holds 
\begin{multline*}
\left|I_{\varepsilon}(\lambda,q)-J_{\varepsilon,g}\left(\tilde{W}_{\lambda\varepsilon^{\frac{1}{4}},q}+\lambda^{2}\varepsilon^{\frac{1}{2}}\tilde{V}_{\lambda\varepsilon^{\frac{1}{4}},q}\right)\right|\\
\le\left\Vert \tilde{\phi}_{\lambda\varepsilon^{\frac{1}{4}},q}\right\Vert _{\mathcal{H}_{g}}^{2}+C\left(\varepsilon\left|\log\varepsilon\right|+\varepsilon^{\frac{1}{2}}\right)\left\Vert \tilde{\phi}_{\lambda\varepsilon^{\frac{1}{4}},q}\right\Vert _{\mathcal{H}_{g}}=o(\varepsilon)
\end{multline*}
$C^{0}$-uniformly for $q\in\partial M$ and $\lambda$ in a compact
set of $(0,+\infty)$.
\end{lem}
The proof of this result is similar to prove of \cite[Lemma 6]{GMP19}
and will be postponed in the Appendix

At this point we can prove the main result of this section.
\begin{prop}
\label{lem:expJeps}Assume $n\ge8$ and $\delta=\lambda\varepsilon^{\frac{1}{4}}$.
It holds 
\[
J_{\varepsilon}(\tilde{W}_{\lambda\varepsilon^{\frac{1}{4}},q}+\lambda^{2}\varepsilon^{\frac{1}{2}}\tilde{V}_{\lambda\varepsilon^{\frac{1}{4}},q})=A+B(\varepsilon)+\varepsilon\lambda^{4}\varphi(q)+C\varepsilon\ln\lambda+o(\varepsilon),
\]
$C^{0}$-uniformly for $q\in\partial M$ and $\lambda$ in a compact
set of $(0,+\infty)$, where

\begin{align*}
A= & \frac{1}{2}\int_{\mathbb{R}_{+}^{n}}|\nabla U(t,z)|^{2}dtdz-\frac{(n-2)^{2}}{2(n-1)}\int_{\mathbb{R}^{n-1}}U(0,z)^{\frac{2(n-1)}{(n-2)}}dz\\
B(\varepsilon)= & \varepsilon\left[\frac{(n-2)^{3}}{2(n-1)}\int_{\mathbb{R}^{n-1}}U^{\frac{2(n-1)}{n-2}}(z,0)dz-\frac{(n-2)^{2}}{2(n-1)}\int_{\mathbb{R}^{n-1}}U^{\frac{2(n-1)}{n-2}}(z,0)\ln U(z,0)dz\right]\\
 & -\varepsilon|\ln\varepsilon|\frac{(n-2)^{3}}{16(n-1)}\int_{\mathbb{R}^{n-1}}U^{\frac{2(n-1)}{n-2}}(z,0)dz\\
\varphi(q)= & \frac{1}{2}\int_{\mathbb{R}_{+}^{n}}v_{q}\Delta v_{q}dtdz-\frac{n-2}{96(n-1)}|\bar{W}(q)|^{2}\int_{\mathbb{R}_{+}^{n}}|z|^{2}U^{2}(t,z)dtdz\\
 & -\frac{(n-2)(n-8)}{2(n^{2}-1)}R_{ninj}^{2}(q)\int_{\mathbb{R}_{+}^{n}}\frac{t^{2}|z|^{4}}{\left((1+t)^{2}+|z|^{2}\right)^{n}}dtdz.\\
C= & \frac{(n-2)^{3}}{4(n-1)}\int_{\mathbb{R}^{n-1}}U^{\frac{2(n-1)}{n-2}}dz>0.
\end{align*}
\end{prop}
\begin{proof}
We write
\begin{multline*}
J_{\varepsilon,g}(\tilde{W}_{\delta,q}+\delta^{2}\tilde{V}_{\delta,q})=J_{0,g}(\tilde{W}_{\delta,q}+\delta^{2}\tilde{V}_{\delta,q})\\
-\frac{(n-2)^{2}}{2(n-1)+\varepsilon(n-2)}\int_{\partial M}\left((\tilde{W}_{\delta,q}+\delta^{2}\tilde{V}_{\delta,q})^{+}\right)^{\frac{2(n-1)}{n-2}+\varepsilon}d\sigma_{g}\\
+\frac{(n-2)^{2}}{2(n-1)}\int_{\partial M}\left((\tilde{W}_{\delta,q}+\delta^{2}\tilde{V}_{\delta,q})^{+}\right)^{\frac{2(n-1)}{n-2}}d\sigma_{g},
\end{multline*}
where 
\begin{align*}
J_{0,g}(v):= & \frac{1}{2}\int_{M}|\nabla_{g}v|^{2}+av^{2}d\mu_{g}+\frac{n-2}{4}\int_{\partial M}h_{g}v^{2}d\sigma_{g}\\
 & -\frac{(n-2)^{2}}{2(n-1)}\int_{\partial M}\left(v^{+}\right)^{\frac{2(n-1)}{n-2}}d\sigma_{g}.
\end{align*}
Since $\frac{(n-2)^{2}}{2(n-1)+\varepsilon(n-2)}=\frac{(n-2)^{2}}{2(n-1)}-\varepsilon\frac{(n-2)^{3}}{2(n-1)}+o(\varepsilon)$
we have
\begin{multline*}
J_{\varepsilon,g}(\tilde{W}_{\delta,q}+\delta^{2}\tilde{V}_{\delta,q})=J_{0,g}(\tilde{W}_{\delta,q}+\delta^{2}\tilde{V}_{\delta,q})\\
-\frac{(n-2)^{2}}{2(n-1)}\int_{\partial M}\left((\tilde{W}_{\delta,q}+\delta^{2}\tilde{V}_{\delta,q})^{+}\right)^{\frac{2(n-1)}{n-2}+\varepsilon}-\left((\tilde{W}_{\delta,q}+\delta^{2}\tilde{V}_{\delta,q})^{+}\right)^{\frac{2(n-1)}{n-2}}d\sigma_{g}\\
+\left[\varepsilon\frac{(n-2)^{3}}{2(n-1)}+o(\varepsilon)\right]\int_{\partial M}\left((\tilde{W}_{\delta,q}+\delta^{2}\tilde{V}_{\delta,q})^{+}\right)^{\frac{2(n-1)}{n-2}+\varepsilon}d\sigma_{g}.
\end{multline*}
For $J_{0,g}$ we proceed as in \cite[Lemma 8]{GMP19} (see also \cite{GM20})
obtaining that 
\[
J_{0,g}(\tilde{W}_{\delta,q}+\delta^{2}\tilde{V}_{\delta,q})=A+\delta^{4}\varphi(q)+o(\delta^{4})
\]
where 
\begin{align*}
A= & \frac{1}{2}\int_{\mathbb{R}_{+}^{n}}|\nabla U(t,z)|^{2}dtdz-\frac{(n-2)^{2}}{2(n-1)}\int_{\mathbb{R}^{n-1}}U(0,z)^{\frac{2(n-1)}{(n-2)}}dz\\
\varphi(q)= & \frac{1}{2}\int_{\mathbb{R}_{+}^{n}}v_{q}\Delta v_{q}dtdz-\frac{n-2}{96(n-1)}|\bar{W}(q)|^{2}\int_{\mathbb{R}_{+}^{n}}|z|^{2}U^{2}(t,z)dtdz\\
 & -\frac{(n-2)(n-8)}{2(n^{2}-1)}R_{ninj}^{2}(q)\int_{\mathbb{R}_{+}^{n}}\frac{t^{2}|z|^{4}}{\left((1+t)^{2}+|z|^{2}\right)^{n}}dtdz.
\end{align*}
Using again (\ref{eq:Uallaeps}) and (\ref{eq:deltaallaeps}), proceeding
similarly to (\ref{eq:feps-f0}), and recalling that $\delta=\lambda\varepsilon^{\frac{1}{4}}$
we have

\begin{multline*}
\int_{\partial M}\left((\tilde{W}_{\delta,q}+\delta^{2}\tilde{V}_{\delta,q})^{+}\right)^{\frac{2(n-1)}{n-2}+\varepsilon}-\left((\tilde{W}_{\delta,q}+\delta^{2}\tilde{V}_{\delta,q})^{+}\right)^{\frac{2(n-1)}{n-2}}d\sigma_{g}\\
=\int_{\partial M}\Lambda_{q}^{\varepsilon}\left((W_{\delta,q}+\delta^{2}V_{\delta,q})^{+}\right)^{\frac{2(n-1)}{n-2}+\varepsilon}-\left((W_{\delta,q}+\delta^{2}V_{\delta,q})^{+}\right)^{\frac{2(n-1)}{n-2}}d\sigma_{\tilde{g}}\\
=(1+o(1))\int_{|z|<\frac{1}{\delta}}\frac{\Lambda_{q}^{\varepsilon}(\delta y)}{\delta^{\varepsilon\frac{n-2}{2}}}\left((U+\delta^{2}v_{q})^{\varepsilon}-1\right)(U+\delta^{2}v_{q})^{\frac{2(n-1)}{n-2}}dz\\
=\int_{\mathbb{R}^{n-1}}\left(-\frac{n-2}{8}\varepsilon\ln\varepsilon-\frac{n-2}{2}\varepsilon\ln\lambda+\varepsilon\ln(U)+o(\varepsilon)\right)U^{\frac{2(n-1)}{n-2}}dz\\
=\frac{n-2}{6}\varepsilon|\ln\varepsilon|\int_{\mathbb{R}^{n-1}}U^{\frac{2(n-1)}{n-2}}dz+\varepsilon\int_{\mathbb{R}^{n-1}}U^{\frac{2(n-1)}{n-2}}\ln(U)dz\\
-\frac{n-2}{2}\varepsilon\ln\lambda\int_{\mathbb{R}^{n-1}}U^{\frac{2(n-1)}{n-2}}dz+o(\varepsilon).
\end{multline*}
Finally, with the same technique, 
\begin{multline*}
\left[\varepsilon\frac{(n-2)^{3}}{2(n-1)}+o(\varepsilon)\right]\int_{\partial M}\left((\tilde{W}_{\delta,q}+\delta^{2}\tilde{V}_{\delta,q})^{+}\right)^{\frac{2(n-1)}{n-2}+\varepsilon}d\sigma\\
=\left[\varepsilon\frac{(n-2)^{3}}{2(n-1)}+o(\varepsilon)\right]\int_{|z|<\frac{1}{\delta}}\frac{\Lambda_{q}(\delta y)}{\delta^{\varepsilon\frac{n-2}{2}}}(U+\delta v_{q})^{\frac{2(n-1)}{n-2}}(U+\delta v_{q})^{\varepsilon}dz+o(\delta^{3})\\
=\varepsilon\frac{(n-2)^{3}}{2(n-1)}\int_{\mathbb{R}^{n-1}}U^{\frac{2(n-1)}{n-2}}+o(\varepsilon).
\end{multline*}
\end{proof}

\section{Proof of Theorem \ref{almaraz}.}

By the following result we prove that once we have a critical point
of the reduced functional $I_{\varepsilon}(\lambda,q)$, we solve
Problem (\ref{eq:Pmain}). The proof of this result is very similar
to the proof of \cite[Claim (i) of Prop. 5]{GMP16}, and we will omit
it for the sake of brevity.
\begin{lem}
\label{lem:punticritici}If $(\bar{\lambda},\bar{q})\in(0,+\infty)\times\partial M$
is a critical point for the reduced functional $I_{\varepsilon}(\lambda,q)$,
then the function $\tilde{W}_{\bar{\lambda}\varepsilon^{\frac{1}{4}},\bar{q}}+\bar{\lambda}^{2}\varepsilon^{\frac{1}{2}}\tilde{V}_{\lambda\varepsilon^{\frac{1}{4}},\bar{q}}+\tilde{\phi}_{\lambda\varepsilon^{\frac{1}{4}},q}$
is a solution of (\ref{eq:Pmain}). Here $\tilde{\phi}_{\lambda\varepsilon^{\frac{1}{4}},q}$
is defined in Proposition \ref{prop:phi}.
\end{lem}
\begin{lem}
\label{lem:phi<0}Assume $n\ge8$ and that the Weyl tensor $W_{g}$
is not vanishing on $\partial M$. Then the function $\varphi(q)$
defined in Proposition \ref{lem:expJeps} is strictly negative on
$\partial M$. 
\end{lem}
\begin{proof}
By Proposition \ref{lem:expJeps} we have that 
\[
\varphi(q)=\frac{1}{2}\int_{\mathbb{R}_{+}^{n}}v_{q}\Delta v_{q}dtdz-C_{1}|\bar{W}(q)|^{2}-(n-8)C_{2}R_{ninj}^{2}(q),
\]
where $C_{1},C_{2}$ are positive constants. We recall (see \cite[page 1618]{M1})
that when the boundary is umbilic $W(q)=0$ if and only if $\bar{W}(q)$
and $R_{ninj}(q)$ are both zero, so, by our assumption, we have that
at least one among $\bar{|W}(q)|$ and $R_{nlnj}^{2}(q)$ which is
strictly positive. This, combined with (\ref{new}), implies that,
$\varphi(q)$ is strictly negative for $n>8$.

When $n=8$ the same strategy leads only to the weak inequality $\varphi(q)\le0$.
To overcome this difficulty, a delicate analysis of the term $\int_{\mathbb{R}_{+}^{n}}v_{q}\Delta v_{q}dtdz$
is needed. An improvement of the estimate (\ref{new}) is performed
in \cite{GMsub}, where the authors give a more precise description
of the function $v_{q}$ as a sum of an harmonic function with explicit
rational functions. This description, for $n=8$, leads to the inequality,
proved in \cite[Lemma 19]{GMsub},
\[
\int_{\mathbb{R}_{+}^{8}}v_{q}\Delta v_{q}dy\le-C_{3}R_{8i8j}^{2}(q),
\]
where $C_{3}>0$. Thus for $n=8$ we have 
\[
\varphi(q)\le-C_{1}|\bar{W}(q)|^{2}-C_{3}R_{8i8j}^{2}(q)<0,
\]
which completes the proof.
\end{proof}
\begin{proof}[Proof of Theorem \ref{almaraz}]
By Lemma \ref{lem:phi<0} we have that the function $\varphi(q)$
defined in Proposition \ref{lem:expJeps} is strictly negative on
$\partial M$. We recall as well, that the number $C$ defined in
the same proposition is positive. Then, defined 
\begin{align*}
I & :[a,b]\times\partial M\rightarrow\mathbb{R}\\
I(\lambda,q) & =\lambda^{2}\varphi(q)+C\log\lambda
\end{align*}
 we have that for any $M<0$ there exist $a,b$ such that 
\[
I(\lambda,q)<M\text{ for any }q\in\partial M,\ \lambda\not\in[a,b]
\]
and
\[
\frac{\partial I}{\partial\lambda}(a,q)\neq0,\ \ \frac{\partial I}{\partial\lambda}(a,q)\neq0\ \forall q\in\partial M.
\]
Then the function $I$ admits a absolute maximum on $[a,b]\times\partial M$.
This maximum is also $C^{0}$-stable. in other words, if $(\lambda_{0},q_{0})$
is the maximum point for $I$, for any function $f\in C^{1}([a,b]\times\partial M)$
with $\|f\|_{C^{0}}$ sufficiently small, then the function $I+f$
on $[a,b]\times\partial M$ admits a maximum point $(\bar{\lambda},\bar{q})$
close to $(\lambda_{0},q_{0})$. 

Then, taken an $\varepsilon$ sufficiently small, in light of Proposition
\ref{lem:JWpiuPhi} and Proposition \ref{lem:expJeps}, there exists
a pair $(\lambda_{\varepsilon},q_{\varepsilon})$ maximum point for
$I_{\varepsilon}(\lambda,q)$. Thus, by Lemma \ref{lem:punticritici},
$v_{\varepsilon}:=\tilde{W}_{\bar{\lambda}\varepsilon^{\frac{1}{4}},\bar{q}}+\bar{\lambda}^{2}\varepsilon^{\frac{1}{2}}\tilde{V}_{\lambda\varepsilon^{\frac{1}{4}},\bar{q}}+\tilde{\phi}_{\lambda\varepsilon^{\frac{1}{4}},q}\in\mathcal{H}_{g}$
is a solution of (\ref{eq:Pmain}). By construction $v_{\varepsilon}$
blows up at $q_{\varepsilon}\rightarrow q_{0}$ when $\varepsilon\rightarrow0$.
\end{proof}

\section{Appendix}
\begin{proof}[Proof of Lemma \ref{lem:JWpiuPhi}]
 We have, for some $\theta\in(0,1)$ 
\begin{multline*}
\tilde{J}_{\varepsilon,\tilde{g}_{q}}(W_{\delta,q}+\delta^{2}V_{\delta,q}+\phi_{\delta,q})-\tilde{J}_{\varepsilon,\tilde{g}_{q}}(W_{\delta,q}+\delta^{2}V_{\delta,q})=\tilde{J}_{\varepsilon,\tilde{g}_{q}}'(W_{\delta,q}+\delta^{2}V_{\delta,q})[\phi_{\delta,q}]\\
+\frac{1}{2}\tilde{J}_{\varepsilon,\tilde{g}_{q}}''(W_{\delta,q}+\delta^{2}V_{\delta,q}+\theta\phi_{\delta,q})[\phi_{\delta,q},\phi_{\delta,q}]\\
=\int_{M}\left(\nabla_{\tilde{g}_{q}}W_{\delta,q}+\delta^{2}\nabla_{\tilde{g}_{q}}V_{\delta,q}\right)\nabla_{\tilde{g}_{q}}\phi_{\delta,q}+\tilde{a}(x)\left(W_{\delta,q}+\delta^{2}V_{\delta,q}\right)\phi_{\delta,q}d\mu_{\tilde{g}_{q}}\\
-(n-2)\int_{\partial M}\Lambda_{q}^{\varepsilon}\left(\left(W_{\delta,q}+\delta^{2}V_{\delta,q}\right)^{+}\right)^{\frac{n}{n-2}+\varepsilon}\phi_{\delta,q}d\sigma_{\tilde{g}_{q}}\\
+\frac{n-2}{2}\int_{\partial M}h_{\tilde{g}_{q}}\left(W_{\delta,q}+\delta^{2}V_{\delta,q}\right)\phi_{\delta,q}d\sigma_{\tilde{g}_{q}}\\
+\frac{1}{2}\int_{M}|\nabla_{\tilde{g}_{q}}\phi_{\delta,q}|^{2}+\tilde{a}(x)\phi_{\delta,q}^{2}d\mu_{\tilde{g}_{q}}+\frac{n-2}{4}\int_{\partial M}h_{\tilde{g}_{q}}\phi_{\delta,q}^{2}d\sigma_{\tilde{g}_{q}}\\
-\frac{n+\varepsilon(n-2)}{2}\int_{\partial M}\Lambda_{q}^{\varepsilon}\left(\left(W_{\delta,q}+\delta^{2}V_{\delta,q}+\theta\phi_{\delta,q}\right)^{+}\right)^{\frac{2}{n-2}+\varepsilon}\phi_{\delta,q}^{2}d\sigma_{\tilde{g}_{q}}.
\end{multline*}
Immediately we have, by definition of $\|\cdot\|_{\tilde{g}}$,
\[
\int_{M}|\nabla_{\tilde{g}_{q}}\phi_{\delta,q}|^{2}+\tilde{a}\phi_{\delta,q}^{2}d\mu_{\tilde{g}_{q}}+\int_{\partial M}\frac{n-2}{4}h_{\tilde{g}_{q}}\phi_{\delta,q}^{2}d\sigma=C\|\phi_{\delta,q}\|_{\tilde{g}_{q}}^{2}=o(\varepsilon).
\]
By Holder inequality one can easily obtain
\[
\int_{M}\tilde{a}W_{\delta,q}\phi_{\delta,q}d\mu_{\tilde{g}_{q}}\le C|W_{\delta,q}|_{L_{\tilde{g}}^{\frac{2n}{n+2}}}|\phi_{\delta,q}|_{L_{\tilde{g}}^{\frac{2n}{n-2}}}\le C\delta^{2}\|\phi_{\delta,q}\|_{\tilde{g}}=o(\varepsilon)
\]
and
\[
\delta^{2}\int_{M}\tilde{a}V_{\delta,q}\phi_{\delta,q}d\mu_{\tilde{g}_{q}}\le C\delta^{2}|V_{\delta,q}|_{L_{\tilde{g}}^{2}}|\phi_{\delta,q}|_{L_{\tilde{g}}^{2}}\le C\delta^{2}\|\phi_{\delta,q}\|_{\tilde{g}}=o(\varepsilon).
\]
In addition, notice that $\left(\left(W_{\delta,q}+\delta^{2}V_{\delta,q}+\theta\phi_{\delta,q}\right)^{+}\right)^{\frac{2}{n-2}+\varepsilon}$
belongs to $L_{\tilde{g}}^{\frac{2(n-1)+n(n-2)\varepsilon}{2+(n-2)\varepsilon}}$and
that $2\left(\frac{2(n-1)+n(n-2)\varepsilon}{2+(n-2)\varepsilon}\right)'=\frac{4(n-1)+2n(n-2)\varepsilon}{2(n-2)+(n-1)(n-2)\varepsilon}<s_{\varepsilon}$,
so, again by Holder inequality we have
\begin{multline*}
\int_{\partial M}\Lambda_{q}^{\varepsilon}\left(\left(W_{\delta,q}+\delta^{2}V_{\delta,q}+\theta\phi_{\delta,q}\right)^{+}\right)^{\frac{2}{n-2}+\varepsilon}\phi_{\delta,q}^{2}d\sigma_{\tilde{g}_{q}}\\
\le C\left(\left|W_{\delta,q}+\delta^{2}V_{\delta,q}+\theta\phi_{\delta,q}\right|_{L_{\tilde{g}}^{s_{\varepsilon}}}^{\frac{2}{n-2}}\right)\|\phi_{\delta,q}\|_{\tilde{g}}^{2}=o(\varepsilon).
\end{multline*}
and by(\ref{eq:hg-resto}) 
\begin{multline*}
\int_{\partial M}h_{\tilde{g}_{q}}(W_{\delta,q}+\delta^{2}V_{\delta,q})\phi_{\delta,q}d\sigma_{\tilde{g}_{q}}\\
\le|h_{\tilde{g}}(W_{\delta,q}+\delta^{2}V_{\delta,q})|_{L_{\tilde{g}}^{\frac{2(n-1)}{n}}(\partial M)}\|\phi_{\delta,q}\|_{\tilde{g}}=O(\delta^{3})\|\phi_{\delta,q}\|_{\tilde{g}}=o(\varepsilon).
\end{multline*}
 By integration by parts we have 
\begin{multline}
\int_{M}\left(\nabla_{\tilde{g}_{q}}W_{\delta,q}+\delta^{2}\nabla_{\tilde{g}_{q}}V_{\delta,q}\right)\nabla_{\tilde{g}_{q}}\phi_{\delta,q}d\mu_{\tilde{g}_{q}}=-\int_{M}\Delta_{\tilde{g}_{q}}\left(W_{\delta,q}+\delta^{2}V_{\delta,q}\right)\phi_{\delta,q}d\mu_{\tilde{g}_{q}}\\
+\int_{\partial M}\left(\frac{\partial}{\partial\nu}W_{\delta,q}+\delta^{2}\frac{\partial}{\partial\nu}V_{\delta,q}\right)\phi_{\delta,q}d\sigma_{\tilde{g}_{q}}.\label{eq:parts}
\end{multline}
and, as in (\ref{eq:Delta-resto}), we get 
\begin{multline*}
\int_{M}\Delta_{\tilde{g}_{q}}\left(W_{\delta,q}+\delta^{2}V_{\delta,q}\right)\phi_{\delta,q}d\mu_{\tilde{g}_{q}}\\
\le|\Delta_{\tilde{g}_{q}}(W_{\delta,q}+\delta^{2}V_{\delta,q})|_{L_{\tilde{g}}^{\frac{2n}{n+2}}}\|\phi_{\delta,q}\|_{\tilde{g}}=O(\delta^{2})\|\phi_{\delta,q}\|_{\tilde{g}}=o(\varepsilon),
\end{multline*}
and for the boundary term in (\ref{eq:parts}), we have, in light
of (\ref{eq:bordo1}), (\ref{eq:bordo2}), (\ref{eq:bordo3}) and
(\ref{eq:bordo4}), that 
\begin{multline*}
\int_{\partial M}\left[\left(\frac{\partial}{\partial\nu}W_{\delta,q}+\delta^{2}\frac{\partial}{\partial\nu}V_{\delta,q}\right)-(n-2)\left(\left(W_{\delta,q}+\delta^{2}V_{\delta,q}\right)^{+}\right)^{\frac{n}{n-2}}\right]\phi_{\delta,q}d\sigma_{\tilde{g}_{q}}\\
=\left|(n-2)\left(\left(W_{\delta,q}+\delta^{2}V_{\delta,q}\right)^{+}\right)^{\frac{n}{n-2}}-\frac{\partial}{\partial\nu}W_{\delta,q}\right|_{L_{\tilde{g}}^{\frac{2(n-1)}{n}}}|\phi_{\delta,q}|_{L_{\tilde{g}}^{\frac{2(n-1)}{n-2}}}\\
=O(\delta^{3})\|\phi_{\delta,q}\|_{\tilde{g}}=o(\varepsilon).
\end{multline*}
At this point it remains to estimate 
\[
\int_{\partial M}\left[\Lambda_{q}^{\varepsilon}\left(\left(W_{\delta,q}+\delta^{2}V_{\delta,q}\right)^{+}\right)^{\frac{n}{n-2}+\varepsilon}-\left(\left(W_{\delta,q}+\delta^{2}V_{\delta,q}\right)^{+}\right)^{\frac{n}{n-2}}\right]\phi_{\delta,q}d\sigma_{\tilde{g}_{q}}
\]
and we proceed as in (\ref{eq:feps-f0}) to get
\begin{multline*}
\int_{\partial M}\left[\Lambda_{q}^{\varepsilon}\left(\left(W_{\delta,q}+\delta^{2}V_{\delta,q}\right)^{+}\right)^{\frac{n}{n-2}+\varepsilon}-\left(\left(W_{\delta,q}+\delta^{2}V_{\delta,q}\right)^{+}\right)^{\frac{n}{n-2}}\right]\phi_{\delta,q}d\sigma_{\tilde{g}_{q}}\\
\le\left|\Lambda_{q}^{\varepsilon}\left(\left(W_{\delta,q}+\delta^{2}V_{\delta,q}\right)^{+}\right)^{\frac{n}{n-2}+\varepsilon}-\left(\left(W_{\delta,q}+\delta^{2}V_{\delta,q}\right)^{+}\right)^{\frac{n}{n-2}}\right|_{L_{\tilde{g}}^{\frac{2(n-1)}{n}}}\|\phi_{\delta,q}\|_{\tilde{g}}=o(\varepsilon),
\end{multline*}
ending the proof.
\end{proof}

\end{document}